\documentclass[preprint,11pt]{elsarticle}

\usepackage[T1]{fontenc}
\usepackage{lmodern}
\usepackage{amsmath,amssymb,amsthm,mathtools,mathrsfs}
\usepackage{geometry}
\usepackage{xcolor}
\usepackage{caption}
\usepackage{tocloft}   
\usepackage{hyperref}  
\makeatletter

\renewcommand{\tmptocnumberline}[1]{%
  \setbox0=\hbox{\appendixname}%
  \appnamewidth=\wd0
  \addtolength{\appnamewidth}{2.5pc}%
  \hb@xt@\appnamewidth{#1.\hfill}}
\makeatother

\hypersetup{colorlinks=true,linkcolor=blue!55!black,urlcolor=blue!55!black}
\allowdisplaybreaks[4]

\newtheorem{theorem}{Theorem}[section]
\newtheorem{lemma}[theorem]{Lemma}
\newtheorem{proposition}[theorem]{Proposition}
\newtheorem{coro}[theorem]{Corollary}
\theoremstyle{remark}

\theoremstyle{ex}

\DeclareMathOperator{\Pf}{Pf}

\newcommand{\ps}{\operatorname{ps}}
\newcommand{\la}{\lambda}
\newcommand{\Ya}{S_{(k,k)}}
\newcommand{\Yb}{S_{(k+1,k)}}
\newcommand{\Yc}{S_{(k+2,k)}}
\newcommand{\Za}{S_{(k,k,1)}}
\newcommand{\Zb}{S_{(k,k,2)}}
\newcommand{\Zc}{S_{(k,k,1,1)}}
\newcommand{\Zd}{S_{(k+1,k,1)}}
\newcommand{\Ze}{S_{(k+1,k,1,1)}}
\newcommand{\Zf}{S_{(k+3,k)}}
\newcommand{\Zg}{S_{(k,k,2,2)}}
\newcommand{\Zh}{S_{(k,k,2,1,1)}}

\newdimen\Squaresize \Squaresize=11pt
\newdimen\Thickness \Thickness=0.7pt
\def\Square#1{\hbox{\vrule width \Thickness
		\vbox to \Squaresize{\hrule height \Thickness\vss
			\hbox to \Squaresize{\hss#1\hss}
			\vss\hrule height\Thickness}
		\unskip\vrule width \Thickness} \kern-\Thickness}

\def\Vsquare#1{\vbox{\Square{$#1$}}\kern-\Thickness}
\def\young#1{\vbox{\smallskip\offinterlineskip \halign{&\Vsquare{##}\cr #1}}}

\begin{document}
\begin{frontmatter}

\title{Pfaffian--Toeplitz identities, Schur positivity, and the
$q$-log-convexity of Baxter polynomials}

\author[]{Yanxin Liu}
\ead{liu-yanxin@outlook.com}

\author[]{Jianxi Mao\corref{cor1}}
\ead{maojx@dlut.edu.cn}
\cortext[cor1]{Corresponding author}

\address{School of Mathematical Sciences, Dalian University of Technology, Dalian 116024,\\ P. R. China}

\begin{abstract}
We use the Pfaffian minor summation formula together with
Jacobi--Trudi Toeplitz matrices to derive Pfaffian
expansions in skew Schur functions.
This yields explicit Schur expansions for several generating
functions involving products of skew Schur functions. 
In particular, using sparse skew-symmetric matrices, we provide a Pfaffian proof of a Schur-positive identity arising in the
study of the $q$-log-convexity of the Narayana polynomials. 
As the main application, we prove that the Baxter polynomials
form a $q$-log-convex sequence. 
We further show that the Baxter transformation defined by the refined Baxter numbers preserves
log-convexity. 
Finally, by realizing the $q$-refined Baxter numbers as
principal specializations of rectangular Schur functions, we prove
that the array of $q$-refined Baxter numbers is $q$-log-concave both along each row and along each
column. 
Both $q$-log-concavity results extend naturally to the $q$-analogues of the $d$-Hoggatt numbers.
\end{abstract}

\end{frontmatter}

\tableofcontents
\noindent\rule{\textwidth}{0.4pt} 
\section{Introduction}
The aim of this article is twofold.
The first is to combine the Pfaffian minor summation formula with skew Jacobi--Trudi Toeplitz matrices to obtain skew Schur expansions of certain symmetric functions. 
The second is to use these Schur-positive identities to prove the $q$-log-convexity of the Baxter polynomials.

A matrix $M=[M_{ij}]_{1\le i,j\le 2n}$ is called \emph{skew-symmetric} if $M_{ij}=-M_{ji}$ for all $i,j$.
The \emph{Pfaffian} of a skew-symmetric matrix $M$, denoted by $\Pf(M)$, is explicitly defined by the formula
\begin{equation*}
	\Pf(M)=\frac{1}{2^n n!}\sum_{\sigma \in \mathfrak{S}_{2n}}\mathrm{sgn}(\sigma)\prod_{i=1}^{n}M_{\sigma(2i-1),\sigma(2i)},
\end{equation*}
where $\mathfrak{S}_{2n}$ is the symmetric group of degree $2n$ and $\mathrm{sgn}(\sigma)$ is the signature of $\sigma$.

Pfaffian methods have numerous applications in the theory of symmetric functions and related areas~\cite{FK01,HIMN17,IN13,IW00,JL15,Oka19,Ste90}.
In particular, Ishikawa and Wakayama~\cite{IW95} established
the Pfaffian minor summation formula.
Let $[m]:=\{1,2,\ldots,m\}$ and let $M_{I,J}$ be the submatrix of $M$ obtained by selecting the rows and columns indexed by $I$ and $J$.
\begin{lemma}{\rm (\cite[Theorem 1]{IW95})}\label{pfcomp}
	Let $m \le n$ and $m$ be even. 
	Suppose that $H$ is an $m\times n$ matrix, and  $\Omega$ is an $n\times n$ 
skew-symmetric matrix. 
	Then $H\Omega H^{\mathsf T}$ is a skew-symmetric matrix and
	\begin{equation}\label{pffor}
		\Pf\left(H\Omega H^{\mathsf T}\right)=
		\sum_{\substack{J \subset [n] \\ |J|=m}}\det H_{[m],J} \Pf\left(\Omega_{J,J}\right).
	\end{equation}
\end{lemma}

This Pfaffian minor summation formula and its specializations have been
used to derive a variety of identities for Schur functions and related symmetric functions~\cite{IW99,IW00,IW06}.
For example,
Sundquist~\cite{Sun96} proved two-variable generalizations of this formula and applied them to express Schur functions in terms of Pfaffians.
Bump and Diaconis~\cite{BD02} further explored the close connection between symmetric functions and Toeplitz matrices.
Okada~\cite{Oka19} gave Pfaffian analogues of the Cauchy--Binet formulas that provide direct proofs of many basic identities for Schur $Q$-functions. 
Huh et al.~\cite{HKKO25} used the Pfaffian minor summation formula to establish affine bounded Littlewood identities for cylindric Schur functions.

Ishikawa, Okada and Wakayama~\cite{IW96} developed a systematic application of the
Pfaffian minor summation formula to identities involving Schur polynomials.
By specializing the matrix $H$ in~\eqref{pffor} to a Vandermonde-type matrix, 
they expressed weighted Schur series as Pfaffians.

In this paper, we develop a Pfaffian method for Schur positivity by combining the
Pfaffian minor summation formula with Jacobi--Trudi Toeplitz matrices.
The Toeplitz formulation allows us to choose sparse
skew-symmetric matrices whose Pfaffians can be evaluated directly.
The minor summation expansion gives an explicitly
Schur-positive series.
Equating the two expressions produces nontrivial Schur-positive
identities. 
As a first illustration, we obtain a Pfaffian proof of the
identity used by Chen, Wang and Yang~\cite{CWY10} in their proof of the
$q$-log-convexity of the Narayana polynomials. 
Further specializations give two additional identities,
and we use them to prove the $q$-log-convexity of the Baxter polynomials.

Let $(a_n)_{n\ge 0}$ be a sequence of nonnegative numbers.  
The sequence
is called \emph{log-concave} (resp. \emph{log-convex}) if 
$
a_n^2\ge a_{n-1}a_{n+1}
$
(resp. $a_n^2\le a_{n-1}a_{n+1}$) for all $n\ge 1$.
Log-concave and log-convex sequences often arise in combinatorics, algebra and geometry; 
see, for example, ~\cite{Bra15,DV08,LW07,Sta89}.
For a polynomial $f(q)\in\mathbb{R}[q]$, write $f(q)\ge_q 0$ if it has only nonnegative coefficients. 
A polynomial sequence
$(f_n(q))_{n\ge 0}$ is called \emph{$q$-log-concave} if
\begin{equation*}
	f_n(q)^2-f_{n-1}(q)f_{n+1}(q) \ge_q 0, \quad n\ge 1,
\end{equation*}
and it is called \emph{$q$-log-convex} if
\begin{equation*}
	f_{n-1}(q)f_{n+1}(q)-f_n(q)^2 \ge_q 0, \quad n\ge 1.
\end{equation*}
The notion of $q$-log-concavity goes back to Stanley.
Sagan~\cite{Sag92_1,Sag92_2} established the $q$-log-concavity of sequences of symmetric functions
and investigated the $q$-log-concavity of the $q$-binomial coefficients.
Liu and Wang~\cite{LW07} introduced the notion of $q$-log-convexity,
and proved that some well-known polynomials such as Bell polynomials and Eulerian polynomials are $q$-log-convex.
The notion of $q$-log-convexity has been extensively studied~\cite{CWY10,CWY11,Zhu14,Zhu21,ZS15}.
In particular, using the theory of Schur positivity,
Chen, Wang and Yang~\cite{CWY10} proved the (strong) $q$-log-convexity of Narayana polynomials.

The principal application in this paper concerns the
Baxter numbers.
Baxter numbers enumerate Baxter permutations, which originated in Baxter's study~\cite{Bax64} of fixed
points for the composite of commuting functions. 
Chung et al.~\cite{CGHK78} presented a sum formula for the Baxter numbers $B_n$, where 
\begin{equation*}
	B_n=\sum_{k=0}^nB(n,k)=\sum_{k=0}^n\frac{\binom{n}{k}\binom{n+1}{k}\binom{n+2}{k}}
	{\binom{k+1}{k}\binom{k+2}{k}}, \quad 0\le k\le n.
\end{equation*}
We call $B(n,k)$ the {\it refined Baxter number}.

The Baxter numbers and Baxter permutations are of great significance in combinatorics, algebra and analysis~\cite{Dil15,LR12,Mal79}. 
Various bijections involving Baxter permutations have been constructed~\cite{Dil12,LL23}.
Fang, Zhang and Zhao~\cite{FZZ26} provided an analytic study of Baxter numbers.
Dilks~\cite{Dil15} developed a comprehensive and unified study of Baxter objects, including 
the involutions, the $(-1)$-phenomenon and gamma-nonnegativity.
It is known~\cite{Cig21} that the binomial coefficients, the Narayana numbers and the refined Baxter numbers can be viewed as specializations of the $d$-Hoggatt numbers with $d=1,2,3$, respectively.
The second author and Shi~\cite{MS26} studied the log-behavior of the $d$-Hoggatt numbers.

The paper is organized as follows.   
In Section 2, we review the necessary background on symmetric functions.
In Section 3, we combine the Pfaffian minor summation formula with skew Jacobi--Trudi Toeplitz matrices to derive several Schur function identities. 
Section 4 introduces the $q$-refined Baxter numbers and interprets refined Baxter numbers and their q-analogues via rectangular Schur specializations.
In Section 5, we establish a Schur positivity theorem and use it to prove the $q$-log-convexity of the Baxter polynomials. 
Section 6 presents several further log-behavior properties associated with the refined Baxter numbers. 
We first show that the Baxter transformation preserves log-convexity. 
Then we prove the $q$-log-concavity of the $q$-refined Baxter numbers both along each fixed row and along each fixed column and extend these results to the $q$-analogues of the $d$-Hoggatt numbers. 
The appendices contain several auxiliary symmetric function identities and computational details used in the proofs.

	\section{Background}\label{sec:background}
	In this section, we give an overview of relevant background on symmetric functions.
	Throughout this paper, we adopt the notation and terminology for partitions
	and symmetric functions in Stanley~\cite{Sta24}. 
	Given a nonnegative integer $n$, 
	a \emph{partition} $\lambda$ of $n$ is a weakly decreasing nonnegative integer sequence
    $(\lambda_1,\lambda_2,\ldots,\lambda_k)$
    such that $\sum_{i=1}^{k}\lambda_i=n$.
	The number of nonzero parts $\lambda_i$ is called the \emph{length} of $\lambda$,
	denoted by $\ell(\lambda)$.
    We write $i^{m_i}$ to indicate that the part $i$ occurs $m_i$ times in $\la$.

	The \emph{Young diagram} of $\lambda$ is a left-justified array of cells with $\ell(\lambda)$ rows and $\lambda_i$ cells in row $i$. 
	Transposing the diagram of $\lambda$, we get the conjugate partition of $\lambda$, denoted $\lambda'$. 
	  A cell $(i,j)$ is the cell in row $i$ and
    column $j$. 
	The \emph{hook length} of $(i,j)$, denoted $h(i,j)$, 
	is given by
	\begin{equation*}
	h(i,j)=\lambda_i+\lambda_j'-i-j+1.
	\end{equation*}
	The \emph{content} of $(i,j)$, denoted $c(i,j)$, is given by
	$c(i,j)=j-i.$
    
	Given two partitions $\lambda$ and $\mu$, we say that $\lambda$ \emph{contains} 
	$\mu$, denoted $\mu\subseteq\lambda$, if $\lambda_i\geq\mu_i$  for each $i$. 
	When $\mu\subseteq\lambda$, we can define a \emph{skew partition} $\lambda/\mu$ as the diagram obtained from the diagram of $\lambda$ by removing the cells in the top-left corner corresponding to the diagram of $\mu$.
	
	A \emph{semistandard Young tableau} of shape $\lambda/\mu$ is an array $T=[T_{ij}]$ of positive integers of shape $\lambda/\mu$ that is weakly increasing in every row and strictly increasing down every column.
	The \emph{type} of $T$ is defined as the composition
	$\alpha=(\alpha_1,\alpha_2,\ldots)$, where $\alpha_i$ is the number of $i$'s in $T$.
	Let $X$ denote the set of variables $\{x_1,x_2,\ldots\}$.
	If $T$ has type $\alpha$, then we write
	\begin{equation*}
	X^T=x_1^{\alpha_1}x_2^{\alpha_2}\cdots.
	\end{equation*}
	
	The \emph{skew Schur function} $s_{\lambda/\mu}(X)$ of shape $\lambda/\mu$ is
	defined as the generating function
	\begin{equation*}
	s_{\lambda/\mu}(X)=\sum_T X^T,
	\end{equation*}
	over all semistandard Young tableaux $T$ of shape $\lambda/\mu$
	filled with positive integers. 
	When $\mu$ is the empty partition $\varnothing$, we call $s_\lambda(X)$ the \emph{Schur function} of shape $\lambda$. 
    We set $s_{\emptyset}(X)=1$. Throughout this paper,
	we adopt the convention that $s_{\lambda}=0$ whenever $\lambda$
	does not define a valid partition.
    
A symmetric function $f$ is called \emph{Schur positive}
if the coefficients $a_{\lambda}$ in the expansion
$
f=\sum_{\lambda} a_{\lambda}s_{\lambda}
$
are all nonnegative,
and we denote $\left[s_{\lambda}\right]f=a_{\la}.$

For a Schur function $s_\lambda$,
define the standard involution $\omega(s_\lambda)=s_{\lambda'}$,
where $\lambda'$ is the conjugate partition of $\lambda$.
Let $f,g$ be two symmetric functions and $a,b\in\mathbb{R}$.
Then 
$$
\omega(af+bg)=a \omega(f)+b\omega(g) \quad \textrm{and} \quad \omega(fg)=\omega(f)\omega(g).
$$
\begin{lemma}\label{SPOMEGA}
Let $f$ be a symmetric function.
Then $f$ is Schur positive if and only if $\omega(f)$ is Schur positive.
\end{lemma}
	Let $Y=\{y_1,y_2,\ldots\}$ be another set of variables, and let
	$s_{\lambda/\mu}(X,Y)$ denote the skew Schur function in $X\cup Y$.
	The following basic property will be used later:
	
	\begin{equation}\label{multi}
		s_{\lambda/\mu}(X,Y)
		=
		\sum_{\nu}
		s_{\lambda/\nu}(X)s_{\nu/\mu}(Y),
	\end{equation}
	where the sum ranges over all partitions $\nu$ satisfying
	$\mu\subseteq\nu\subseteq\lambda$; see~\cite[Section 7]{Sta24}.

Let 
$$
e_k(X)=\sum_{i_1<\ldots<i_k} x_{i_1}\dots x_{i_k}
$$
be the \emph{elementary symmetric functions} with $e_0=1$ and $e_k=0$ for $k<0$; and let
$$
h_k(X)=\sum_{i_1\le\ldots\le i_k} x_{i_1}\dots x_{i_k}
$$
be the \emph{complete homogeneous symmetric functions}
with $h_0=1$ and $h_k=0$ for $k<0$.
Then $$e_k(X)=s_{(1^k)}(X) \quad \textrm{and} \quad h_k(X)=s_{(k)}(X).$$
The \emph{skew Jacobi--Trudi identity} (see~\cite[Theorem 7.16.1]{Sta24}) expresses the skew Schur function
$s_{\lambda/\mu}$ as a determinant of complete homogeneous symmetric
functions,
\begin{equation}\label{eq:Jacobi-Trudi}
	s_{\lambda/\mu}=\det(h_{\lambda_i-\mu_j-i+j})_{1\le i,j\le \ell(\la)}.
\end{equation}
Note that for a two-row partition $\la=(a,b)$ and $\mu=\varnothing$, where $a\geq b$,
the Jacobi--Trudi identity gives
\begin{equation}\label{eq:two-row-JT}
	s_{(a,b)}
	=
	\det
	\begin{bmatrix}
		h_a & h_{a+1}\\
		h_{b-1} & h_b
	\end{bmatrix}
	=
	h_a h_b-h_{a+1}h_{b-1}.
\end{equation}

Recall that a \emph{lattice permutation} of length $n$ is a sequence
$w_1w_2\cdots w_n$ such that for any $i$ and $j$ in the subsequence
$w_1w_2\cdots w_j$ the number of $i$'s is greater than or equal to
the number of $i+1$'s. 
Let $T$ be a semistandard Young tableau. 
The \emph{reverse reading word} $T^{\mathrm{rev}}$ is a sequence of entries of
$T$ obtained by first reading each row from right to left and then
concatenating the rows from top to bottom. 
If the reverse reading word $T^{\mathrm{rev}}$ is a lattice permutation, we call $T$ a \emph{Littlewood--Richardson tableau}. 
The \emph{Littlewood--Richardson rule} enables us to expand a product of Schur functions in terms of Schur functions. 
Given two Schur functions $s_{\mu}$ and $s_{\nu}$, Littlewood--Richardson coefficients
$c_{\mu,\nu}^\lambda$ are defined by the following relation
\begin{equation}\label{LRcoeff}
s_{\mu}s_{\nu}
=
\sum_{\lambda} c_{\mu,\nu}^{\lambda}s_{\lambda},
\end{equation}
where $c_{\mu,\nu}^{\lambda}$ is the number of Littlewood--Richardson tableaux of shape $\lambda/\mu$ and content $\nu$~\cite[A1.5]{Sta24}.
For example, if $\lambda=(4,3,1),\,\mu=(2,1),\,\nu=(3,2)$, then
$c_{\mu,\nu}^{\lambda}=2$. Indeed, there are two
Littlewood--Richardson tableaux of shape $\lambda/\mu$ and content $\nu$
as shown in Figure \ref{fig-1}.

\begin{figure}[h]
$$
\young{ * & * & 1 & 1 \cr
* & 1 & 2 \cr
2 \cr
}
\qquad\qquad
\young{ * & * & 1 & 1 \cr
* & 2 & 2 \cr
1 \cr
}
$$
\caption{Littlewood--Richardson tableaux for $\lambda=(4,3,1)$, $\mu=(2,1)$, $\nu=(3,2)$}
\label{fig-1}
\end{figure}

When $\nu=(k)$ or $\nu=(1^k)$ in~\eqref{LRcoeff}, the Littlewood--Richardson rule has
a simpler description, known as \emph{Pieri's rule}. 
We need the notion of horizontal and vertical strips. 
A skew partition $\lambda/\mu$ is called a \emph{horizontal} (or \emph{vertical}) \emph{strip} of size $k$ if there are $k$ cells in total with no two cells lying in the same column (resp.\ in the same row).

\begin{theorem}{\rm \cite[Theorem 7.15.7, Corollary 7.15.9]{Sta24}}
	\label{h-e}
	We have
	\begin{equation*}
		s_{\mu}h_k=\sum_{\lambda}s_{\lambda},
	\end{equation*}
	summed over all partitions $\lambda$ such that $\lambda/\mu$ is a horizontal strip of size $k$, and
	\begin{equation*}
		s_{\mu}e_k=\sum_{\lambda}s_{\lambda},
	\end{equation*}
	summed over all partitions $\lambda$ such that $\lambda/\mu$ is a vertical strip of size $k$.
\end{theorem}

\section{Main results}\label{sec:main}
In this section, we combine the Pfaffian minor summation formula~\eqref{pffor} with Jacobi--Trudi Toeplitz matrices
to derive some identities involving symmetric functions.
We first recall some basic properties of $\Pf(M)$; see~\cite{Knu96}.	
	\begin{lemma}\label{lem-pf}
		Let $M=[M_{ij}]_{1\le i,j\le 2n}$ be a skew-symmetric matrix.
		\begin{itemize}
			\item[\rm(i)] For a fixed row index $i$,
			if $M_{ij}=0$ for $1\le j\le 2n,$ then $\Pf(M) = 0$.
			\item[\rm(ii)] The first-row expansion formula for the Pfaffian is
			\begin{equation*}
				\Pf(M)=\sum_{j=2}^{2n}
				(-1)^j M_{1j}
				\Pf\left(M_{[2n]\setminus\{1,j\},[2n]\setminus\{1,j\}}\right).
			\end{equation*}
		\end{itemize}
	\end{lemma}

Throughout this section, the row and column indices of all matrices start at \(1\), and we use 
$\{j_1,j_2,\cdots,j_m\}_<$ to denote
the set satisfying $j_1<j_2<\cdots<j_m$.
Recall that $h_k$
is the $k$-th complete homogeneous symmetric function.
Let $m$ be a positive integer and $\mu$ be a partition with $\ell(\mu)\le m$.
Define the row--truncated Toeplitz
matrix
\begin{equation*}
H^{(m)}_\mu=\bigl[h_{j-i-\mu_{m-i+1}}\bigr]_{1\le i\le m,\ j\ge1}.
\end{equation*}
Define $H^{(m)}= H^{(m)}_\varnothing.$
The following is a corollary of the skew Jacobi--Trudi identity.
\begin{lemma}\label{lem:skew-toeplitz-minor}
	Let $J=\{j_1,\ldots,j_m\}_<$. Then we have
	\begin{equation*}
	\det\left(H^{(m)}_\mu\right)_{[m],J}=s_{\lambda/\mu},
	\end{equation*}
	where $\lambda=\bigl(j_m-m,j_{m-1}-(m-1),\ldots,j_1-1\bigr)$.
\end{lemma}

\begin{proof}
Let $H=H_\mu^{(m)}$.
By definition,
$$
H_{[m],\{j_1,j_2,\cdots,j_m\}}
=
\begin{bmatrix}
	h_{j_1-1-\mu_m}     & h_{j_2-1-\mu_m} & \cdots & h_{j_m-1-\mu_m}\\
	h_{j_1-2-\mu_{m-1}} &  h_{j_2-2-\mu_{m-1}}   & \cdots & h_{j_m-2-\mu_{m-1}}\\
	\vdots & \vdots & \ddots   & \vdots\\
	h_{j_1-m-\mu_1} & h_{j_2-m-\mu_1} &\vdots & h_{j_m-m-\mu_1}
\end{bmatrix}.
$$
After reversing the orders of both the rows and columns and then transposing the matrix, we obtain, by the skew Jacobi--Trudi identity~\eqref{eq:Jacobi-Trudi},
$$
\det H_{[m],\{j_1,j_2,\cdots,j_m\}}
=\det\begin{bmatrix}
h_{j_m-m-\mu_1}     & h_{j_m-m+1-\mu_2} & \cdots & h_{j_m-1-\mu_m}\\
h_{j_{m-1}-m-\mu_1}  & h_{j_{m-1}-m+1-\mu_2} & \cdots & h_{j_{m-1}-1-\mu_m}\\
\vdots & \vdots & \ddots   & \vdots\\
h_{j_1-m-\mu_1} & h_{j_1-m+1-\mu_2} &\vdots & h_{j_1-1-\mu_m}
\end{bmatrix}=s_{\lambda/\mu},
$$
where $\la_k=j_{m-k+1}-(m-k+1)$ for $k=1,2,\dots,m$.
\end{proof}

\subsection{Toeplitz--Pfaffian Schur expansions}

Let $\mathbf v=(v_1,v_2,\ldots)$ be commuting indeterminates, and define
the infinite skew-symmetric matrix $\Omega(\mathbf v)$ by
\begin{equation*}
\Omega(\mathbf v)_{i,i+1}=v_i,\quad
\Omega(\mathbf v)_{i+1,i}=-v_i\quad (i\ge1),
\end{equation*}
with all other entries equal to zero.
Throughout this subsection, all infinite sums are interpreted
coefficientwise with respect to the usual grading of the ring of
symmetric functions.

\begin{theorem}
	\label{thm:weighted-skew-pfaffian}
	For $n\ge1$ and $\ell(\mu)\le2n$, we have
	\begin{equation}\label{eq:weighted-skew-pfaffian}
		\Pf\left(H^{(2n)}_\mu\Omega(\mathbf v)
		\bigl(H^{(2n)}_\mu\bigr)^{\mathsf T}\right)
		=
		\sum_{0\le a_1\le\cdots\le a_n}
		s_{(a_n^2,a_{n-1}^2,\ldots,a_1^2)/\mu}
		\prod_{r=1}^n v_{a_r+2r-1}.
	\end{equation}
\end{theorem}

\begin{proof}
For $N\geq2n$, let $H= H^{(2n)}_\mu$ and let $H_{[2n],[N]}$ be the submatrix of $H^{(2n)}_\mu$ with the first $N$ columns.
Let $\Omega_N(\mathbf v)$ be the $N$-th leading principal submatrix of $\Omega(\mathbf v)$.
Lemma~\ref{pfcomp} gives
\begin{equation}
\Pf\left(H_{[2n],[N]} \Omega_N (\mathbf v)\left(H_{[2n],[N]}\right)^{\mathsf T}\right)=
\sum_{\substack{J\subseteq[N]\\|J|=2n}}
\det H_{[2n],J}\,\Pf\left(\left(\Omega_N(\mathbf v)\right)_{J,J}\right).
\label{eq:first-finite-expansion}
\end{equation}

Set $J=\{j_1,\cdots,j_{2n}\}_<$.
Since the only nonzero entries of $\Omega(\mathbf v)$ join consecutive indices, its Pfaffian is nonzero exactly
when
$
j_{2r}=j_{2r-1}+1,
$
for $1\leq r\leq n.$
In this case,
\begin{equation*}
J=\{j_1,j_1+1, j_3,j_3+1,\cdots,j_{2n-1},j_{2n-1}+1\}
\end{equation*}
and 
$$
\left(\Omega_N(\mathbf v)\right)_{J,J}
	=
	\begin{bmatrix}
		0     & v_{j_1} & &&&&  \\
		 -v_{j_1}      & 0  & \ast&&&& \\
		 & \ast & 0  & v_{j_3}&&&\\
		  & &-v_{j_3}&0 &\ast & &\\
        &&&\ddots&\ddots&\ddots&\\
        &&&&\ast&0&v_{j_{2n-1}}\\
        &&&&&-v_{j_{2n-1}}&0
	\end{bmatrix}.
$$
By the definition of the Pfaffian, we have
$$
\Pf\left(\left(\Omega_N(\mathbf v)\right)_{J,J}\right)=\prod_{r=1}^n v_{j_{2r-1}}.
$$
Let
$
a_r=j_{2r-1}-(2r-1).
$
Then we have $0\leq a_1\leq\cdots\leq a_n$, and
$$
\Pf\left(\left(\Omega_N(\mathbf v)\right)_{J,J}\right)=\prod_{r=1}^n v_{a_r+2r-1}.
$$

Note that
$$
j_{2r-1}=a_r+2r-1,\quad j_{2r}=a_r+2r.
$$
By Lemma~\ref{lem:skew-toeplitz-minor}, we obtain  $\det H_{[2n],J}=s_{\lambda/\mu},$ where
$$
\lambda=(j_{2n-1}+1-2n,j_{2n-1}-(2n-1),\ldots,j_1-1)=(a_n,a_n,a_{n-1},a_{n-1},\ldots,a_1,a_1).
$$
For every fixed homogeneous degree, only finitely many tuples
$(a_1,\ldots,a_n)$ contribute.
Hence both sides stabilize coefficientwise as $N\to\infty$.
Substitution into \eqref{eq:first-finite-expansion} and letting $N\to\infty$ proves \eqref{eq:weighted-skew-pfaffian}.
\end{proof}


Recall that $H^{(2n)}=H^{(2n)}_\varnothing$.
The following is an immediate consequence of Theorem~\ref{thm:weighted-skew-pfaffian}.
\begin{coro}\label{thm:PfM}
Let $\mathbf v=(v_1,v_2,\ldots)=(1,t,t^2,\ldots)$,
i.e., $v_i=t^{i-1}$.
Then for every $n\geq 1$, 
\begin{equation*}
\Pf\left(H^{(2n)}\Omega(\mathbf v)\left(H^{(2n)}\right)^{\mathsf T}\right)=
\sum_{0\leq a_1\leq\cdots\leq a_n}
s_{(a_n^2,a_{n-1}^2,\ldots,a_1^2)}\,
t^{n(n-1)+\sum_{r=1}^n a_r}.
\end{equation*}
\end{coro}

Define
$\Omega^{(p)}(\mathbf v)$ by
\begin{equation}\label{def:Omega}
\Omega^{(p)}(\mathbf v)_{i,i+1}=v_i,\quad
\Omega^{(p)}(\mathbf v)_{i+1,i}=-v_i\quad (i\ge p),
\end{equation}
with all other entries equal to zero.
In other words,
$\Omega^{(p)}(\mathbf v)=\Omega(\mathbf v)$ with $v_1=\cdots=v_{p-1}=0.$
For $n\ge 0$  and $p\ge 2$,
define
\begin{equation*}
\widetilde H^{(2n+1)}_\mu =
\begin{bmatrix}
1&0\\
0&H^{(2n+1)}_\mu
\end{bmatrix},
\quad
\widetilde\Omega^{(p)}(\mathbf v) =
\begin{bmatrix}
0&w_{p-1}\\
-w_{p-1}^{\mathsf T}&\Omega^{(p)}(\mathbf v)
\end{bmatrix},
\end{equation*}
where $w_r=(0,\dots,0,1,0,\dots)$ is a row vector with a $1$ at the $r$-th position.
For example, 
\begin{equation*}
\widetilde H^{(3)}=
\begin{bmatrix}
1 &0 & 0& 0& \cdots& \\
& h_0 & h_1 & h_2 &\cdots \\
& & h_0 & h_1 & \cdots \\
& & & h_0 & \cdots \\
\end{bmatrix},
\quad 	
\widetilde\Omega^{(3)}(\mathbf v)=
\begin{bmatrix}
0      & 0      & 1      & 0   & 0   & 0   & \cdots \\
0      & 0      &        &        &        &        &        \\
-1    &        & 0      &        &        &        &        \\
0  &        &        & 0      & v_3      &        &        \\
0   &        &   & -v_3     & 0      &v_4      &        \\
0   &        &        &        & -v_4    & 0      & \ddots \\
\vdots &        &        &        &        & \ddots & \ddots
\end{bmatrix}.
\end{equation*}
Note that $\widetilde H^{(2n+1)}_\mu \widetilde\Omega^{(p)}(\mathbf v) \bigl(\widetilde H^{(2n+1)}_\mu\bigr)^{\mathsf T}$ equals
\begin{align*}
\begin{bmatrix}
	0&w_{p-1}(H^{(2n+1)}_\mu)^{\mathsf T}\\
			-H^{(2n+1)}_\mu w_{p-1}^{\mathsf T}& H^{(2n+1)}_\mu \Omega^{(p)}(\mathbf v) \bigl( H^{(2n+1)}_\mu\bigr)^{\mathsf T}
		\end{bmatrix}.
\end{align*}

\begin{theorem}
	\label{thm:augmented-skew-pfaffian}
	For $n\ge0$, $p\ge2$, and $\ell(\mu)\le2n+1$, we have
	\begin{equation}\label{eq:fully-weighted-augmented-skew-pfaffian}
		\Pf\left(\widetilde H^{(2n+1)}_\mu
		\widetilde\Omega^{(p)}(\mathbf v)
		\bigl(\widetilde H^{(2n+1)}_\mu\bigr)^{\mathsf T}\right)
		=
		\sum_{0\leq a_1\leq\cdots\leq a_n}
s_{\bigl((a_n+p-2)^2,\ldots,(a_1+p-2)^2,p-2\bigr)/\mu}\prod_{r=1}^n v_{a_r+p+2r-2}.
	\end{equation}
\end{theorem}

\begin{proof}
For $N\geq \max\{2n+2,p\}$, 
let $\widetilde H= \widetilde H^{(2n+1)}_\mu$ and let $\widetilde H_{[2n+2],[N]}$ be the submatrix of $\widetilde H^{(2n+1)}_\mu$ with the first $N$ columns.
Let $\widetilde \Omega^{(p)}_N(\mathbf v)$ be the $N$-th leading principal submatrix of $\widetilde\Omega^{(p)}(\mathbf v)$.
Lemma~\ref{pfcomp} gives
\begin{equation}
\Pf\left(\widetilde H_{[2n+2],[N]} \widetilde\Omega^{(p)}_N (\mathbf v) \left(\widetilde H_{[2n+2],[N]}\right)^T\right)=
\sum_{\substack{J\subseteq[N]\\|J|=2n+2}}
\det \widetilde H_{[2n+2],J}\,\Pf\left(\left(\widetilde\Omega^{(p)}_N(\mathbf v)\right)_{J,J}\right).
\label{eq:second-finite-expansion}
\end{equation}

Set $J=\{j_1,j_2,\cdots,j_{2n+2}\}_<$.
Since the first row of $\widetilde H^{(2n+1)}$ has its only nonzero entry in the first column, 
$$
\det \widetilde H_{[2n+2],J}=0
$$
unless $j_1=1$.
Assume henceforth that $j_1=1$.
By the definition of
$\widetilde{\Omega}^{(p)}_N(\mathbf{v})$, its first row is nonzero only
in column $p$. 
Consequently,
$$\Pf\left(\left(\widetilde\Omega^{(p)}_N(\mathbf v)\right)_{J,J}\right)=0$$
unless $j_2=p$.
Suppose that $j_1=1,j_2=p$. For the remaining indices, a necessary condition for a nonzero Pfaffian is that $j_{2r+2}=j_{2r+1}+1$  for all $ 1\leq r\leq n$.
In this case, we have
\begin{equation*}
J=\{1,p, j_3,j_3+1,j_5,j_5+1,\dots, j_{2n+1},j_{2n+1}+1\}
\end{equation*}
and 
$$
\left(\widetilde\Omega^{(p)}_N(\mathbf v)\right)_{J,J}
	=
	\begin{bmatrix}
    0  & 1 & &&&&  \\ 
    -1 & 0  & \ast&&&& \\
    & \ast & 0  & v_{j_3-1}&&&\\ 
    & &-v_{j_3-1}&0 &\ast & &\\         &&&\ddots&\ddots&\ddots&\\         &&&&\ast&0&v_{j_{2n+1}-1}\\
    &&&&&-v_{j_{2n+1}-1}&0 
    \end{bmatrix}.
$$
By the definition of the Pfaffian,  we have
$$
\Pf\left(\left(\widetilde\Omega^{(p)}_N(\mathbf v)\right)_{J,J}\right)
=\prod_{r=1}^n v_{j_{2r+1}-1}.
$$
Set $a_r=j_{2r+1}-p-2r+1$. Then we have $0\leq a_1\leq a_2\leq\cdots\leq a_n$, 
and
$$
\Pf\left(\left(\widetilde\Omega^{(p)}_N(\mathbf v)\right)_{J,J}\right)
=\prod_{r=1}^n v_{a_r+p+2r-2}.
$$

Note that
$$
j_{2r+1}=a_r+p+2r-1,\quad
j_{2r+2}=a_r+p+2r.
$$
By Lemma~\ref{lem:skew-toeplitz-minor}, we obtain 
$$
\det\left(\widetilde H_{[2n+2],J}\right)
=
\det\left( H^{(2n+1)}_{\mu}\right)_{[2n+1],
\{p-1,j_3-1,\ldots,j_{2n+1}-1,\,j_{2n+1}\}}
=s_{\lambda/\mu},
$$
where the partition $\la$ equals
\begin{align*}
\left( j_{2n+1}-(2n+1),\, j_{2n+1}-(2n+1),\,\ldots,\, p-2 \right)
=\left((a_n+p-2)^2,\ldots,(a_1+p-2)^2,p-2\right).
\end{align*}
For every fixed homogeneous degree, only finitely many tuples
$(a_1,\ldots,a_n)$ contribute.
Hence both sides stabilize coefficientwise as $N\to\infty$.
Substituting into \eqref{eq:second-finite-expansion} and letting $N\to\infty$ proves \eqref{eq:fully-weighted-augmented-skew-pfaffian}.
\end{proof}

The following is an immediate consequence of Theorem~\ref{thm:augmented-skew-pfaffian}.

\begin{coro}\label{thm:augmented}
Let $\mathbf v=(v_1,v_2,\ldots)$ with $v_1=\cdots=v_{p-1}=0$ and $v_i=t^{i-p}$ for $i\ge p.$
Let $\widetilde H^{(2n+1)}=\widetilde H^{(2n+1)}_\varnothing.$
For every $n\geq 0$ and $p\geq2$, we have
\begin{equation}
\Pf\left(\widetilde H^{(2n+1)}\widetilde\Omega^{(p)}(\mathbf v)
\left(\widetilde H^{(2n+1)}\right)^{\mathsf T}\right)=
\sum_{0\leq a_1\leq\cdots\leq a_n}
s_{\bigl((a_n+p-2)^2,\ldots,(a_1+p-2)^2,p-2\bigr)}
\,t^{n(n-1)+\sum_{r=1}^n a_r}.
\label{eq:augmented-pf}
\end{equation}
\end{coro}

\subsection{Applications}
In this subsection,
we present some applications of our main results.
We first need the following lemma.
Recall that $H^{(n)}=H^{(n)}_\varnothing=\bigl[h_{j-i}\bigr]_{1\le i\le n,\ j\ge1}$.

\begin{lemma}\label{lem:entry}
Given $p\ge1$, let $\Omega^{(p)}(\mathbf v)$ be an infinite skew-symmetric matrix defined in~\eqref{def:Omega} with 
$v_i=0$ for $i<p$  and $v_i=t^{i-p}$ for $i\ge p.$
Let
\begin{equation*}
M^{(n,p)}
=
H^{(n)}\Omega^{(p)}(\mathbf v)
\bigl(H^{(n)}\bigr)^{\mathsf T}.
\end{equation*}
Then $M^{(n,p)}$ is an $n\times n$ skew-symmetric matrix.
For
$1\leq i<j\leq n$, we have
\begin{align}
M^{(n,p)}_{ij}
=\sum_{k\ge 0}
s_{(k+p-i,k+p+1-j)}\,t^k.
\label{eq:general-entry-shifted}
\end{align}
\end{lemma}

\begin{proof}
Direct computation gives
$$
\bigl(M^{(n,p)}\bigr)^{\mathsf T}
=
H^{(n)}
\bigl(\Omega^{(p)}(\mathbf v)\bigr)^{\mathsf T}
\bigl(H^{(n)}\bigr)^{\mathsf T}
=
-M^{(n,p)},
$$
since $\Omega^{(p)}(\mathbf v)$ is skew-symmetric.
Thus, $M^{(n,p)}$ is an $n\times n$ skew-symmetric matrix.

By matrix multiplication,
$$
M^{(n,p)}_{ij}
=
\sum_{a,b\geq1}
h_{a-i}\,\left(\Omega^{(p)}(\mathbf v)_{a,b}\right)\, h_{b-j}.
$$
The only nonzero entries of $\Omega^{(p)}(\mathbf v)$ are
$$
\Omega^{(p)}(\mathbf v)_{r,r+1}=t^{r-p},
\quad
\Omega^{(p)}(\mathbf v)_{r+1,r}=-t^{r-p},
\quad r\geq p.
$$
It follows that
$$
M^{(n,p)}_{ij}
=
\sum_{r\geq p}
\left(
h_{r-i}h_{r+1-j}
-
h_{r+1-i}h_{r-j}
\right)\,t^{r-p}.
$$
By the two-row Jacobi--Trudi identity~\eqref{eq:two-row-JT},
$$
h_{r-i}h_{r+1-j}
-
h_{r+1-i}h_{r-j}
=
s_{(r-i,r+1-j)}.
$$
Replacing $r$ by
$k+p$ yields \eqref{eq:general-entry-shifted}.
\end{proof}

Note that  for a $4\times 4$ skew-symmetric matrix $M$, 
	\begin{equation*}
		\Pf(M)=M_{12}M_{34}-M_{13}M_{24}+M_{14}M_{23}.
	\end{equation*}	
For any family of partitions $\lambda(k)$ depending on $k$, we write
\begin{equation*}
S_{\lambda(k)}(t):=\sum_{k\geq 0}s_{\lambda(k)}t^k,
\end{equation*}
and we suppress the argument $t$ whenever no confusion can arise.	
Here all symmetric function generating series identities involving $t$
are interpreted in the formal power-series ring of symmetric functions over $\mathbb Q$.
We use $S_\la^2$ to denote $S_\la S_\la.$
For example, equation~\eqref{eq:general-entry-shifted} can be written as
$$
M^{(n,p)}_{ij}
=S_{(k+p-i,k+p+1-j)}.
$$

The following identity is equivalent to
Chen, Wang and Yang~\cite[Theorem~6.7]{CWY10} after applying
the standard involution $\omega$.

\begin{theorem}\label{thm:base-pf}
    We have 
    \begin{equation}
		t\cdot\left(\Ya\Yc-\Yb^2\right)+\Ya^2
		=\sum_{0\le b\le a}s_{(a,a,b,b)}\,t^{a+b}.
		\label{eq:base-pf}
	\end{equation}
\end{theorem}
\begin{proof}
Let $\mathbf v=(v_1,v_2,\ldots)=(1,t,t^2,\ldots)$
and $\Omega^{(1)}=\Omega^{(1)}(\mathbf v)$.
Let $M=M^{(4,1)}=H^{(4)}\Omega^{(1)}\left(H^{(4)}\right)^{\mathsf T}$. 
By Lemma~\ref{lem:entry},
		$$
		M_{ij}=\sum_{k\geq1}
		s_{(k-i,k+1-j)}\,t^{k-1}
		$$
		for $i<j$. Consequently, $M_{12}=\Ya,$ $M_{13}=t\cdot \Yb$,
		\begin{align*}
			M_{14}=t^2\cdot \Yc,\quad
			M_{23}=t\cdot \Ya,\quad
			M_{24}=t^2\cdot \Yb,\quad
			M_{34}=t^2\cdot\Ya.
		\end{align*}
Hence, the Pfaffian $\Pf(M)$ equals 
\begin{align*}
	M_{12}M_{34}-M_{13}M_{24}+M_{14}M_{23}
			=t^3\cdot\left(\Ya\Yc-\Yb^2\right)+t^2\cdot\Ya^2.
		\end{align*}
By Corollary~\ref{thm:PfM},
$$\Pf(M)=\sum_{0\le b\le a}s_{(a,a,b,b)}\,t^{a+b+2}.$$
Equation~\eqref{eq:base-pf} follows by comparing the two expressions of $\Pf(M)$.
\end{proof}

Chen, Wang and Yang~\cite{CWY10} proved the $q$-log-convexity of the Narayana polynomials.
Following their notation, one of the key steps is to prove the Schur positivity of $D(r)=\sum_{k=0}^{r}D(r,k,0,0),$ where
$$
D(r,k,0,0)
=s_{(2^k)}\,s_{(2^{r-k-1})}+s_{(2^{k-1},1^2)}\,s_{(2^{r-k-1})}-s_{(2^{k-1},1)}\,s_{(2^{r-k-1},1)}.
$$
See, for instance,~\cite[Section 3]{CWY10}.
        
\begin{coro}{\rm(\cite[Theorem 3.1]{CWY10})}
For $r\ge 1$, the symmetric function $D(r)$ is Schur positive.
\end{coro}
\begin{proof}
By Lemma~\ref{SPOMEGA},
it suffices to prove the Schur positivity of the standard involution
$$
\omega(D(r))=\sum_{k=0}^{r}\left(
s_{(k,k)}\,s_{(r-k-1,r-k-1)}+s_{(k+1,k-1)}\,s_{(r-k-1,r-k-1)}-s_{(k,k-1)}\, s_{(r-k,r-k-1)}
\right).
$$
By the convolution formula for the coefficients of formal power series,
\[
S_{(k,k)}^2=\sum_r t^r\left(\sum_{k=0}^{r}s_{(k,k)}s_{(r-k,r-k)}\right) 
=\sum_r t^{r-1}\left(\sum_{k=0}^{r-1}s_{(k,k)}s_{(r-k-1,r-k-1)}\right)
\]
Applying the same  argument to the other two terms shows
that $\omega(D(r))$ equals the coefficient of $t^{r-1}$ in 
$$
t\cdot \left(\Ya\Yc-\Yb^2\right)+\Ya^2.
$$
By Theorem~\ref{thm:base-pf},
$$
t\cdot\left(\Ya\Yc-\Yb^2\right)+\Ya^2=\sum_{0\le b \le a}s_{(a,a,b,b)}\,t^{a+b}.
$$
		Taking $a=r-q-1$ and $b=q$, we have 
		$$
		\omega(D(r))=\sum_{0\le q \le r-q-1}s_{(r-q-1,r-q-1,q,q)}=\sum_{q=0}^{\lfloor{\frac{r-1}{2}\rfloor}}s_{(r-q-1,r-q-1,q,q)},
		$$
		and the Schur positivity of $D(r)$ follows.
	    \end{proof}

As a special case, 
setting $n=2$ in Corollary~\ref{thm:augmented} yields the Schur expansion of two further symmetric-function generating series.
These identities are precisely the forms
needed in the proof of the Baxter Schur-positivity theorem in
Section~5.
\begin{theorem}\label{thm:base-pf22}
We have
    \begin{align}
		&t\cdot\left(\Yb\Yc-\Ya\Zf\right)-\Ya\Yb\nonumber\\
        &=\sum_{0\le b\le a}
		s_{(a+1,a+1,b+1,b+1,1)}\,t^{a+b+2}-e_1 \sum_{0\le b\le a}s_{(a,a,b,b)}\,t^{a+b}
		\label{eq:base-pf2}
	\end{align}
and 
\begin{equation}\label{jr3}
		\Ya\Zg-\Zc^2+S_{(k,k,1)}S_{(k,k,1,1,1)}=\sum_{0\le b\le a}s_{(a+2,a+2,b+2,b+2,2)}\,t^{a+b+3}.
	\end{equation}
\end{theorem}
 
\begin{proof}[Proof of \eqref{eq:base-pf2}]
Let $\mathbf v=(v_1,v_2,\ldots)$ with $v_1=v_{2}=0$ and $v_i=t^{i-3}$ for $i\ge 3.$
Let $\Omega^{(3)}=\Omega^{(3)}(\mathbf v)$.
Take $n=2, p=3$ in \eqref{eq:augmented-pf} and let $\widetilde M=\widetilde H^{(5)}\widetilde\Omega^{(3)}
\bigl(\widetilde H^{(5)}\bigr)^{\mathsf T}$.
Then 
\begin{align*}
\widetilde M&=\begin{bmatrix}
1&0\\
0&H^{(5)}
\end{bmatrix}
\begin{bmatrix}
0&w_{2}\\
-w_{2}^{\mathsf T}&\Omega^{(3)}
\end{bmatrix}
\begin{bmatrix}
1&0\\
0&(H^{(5)})^{\mathsf T}
\end{bmatrix}
=\begin{bmatrix}
			0&w_2(H^{(5)})^{\mathsf T}\\
			-H^{(5)} w_2^{\mathsf T}& M
		\end{bmatrix},
\end{align*}
where $M=H^{(5)}\Omega^{(3)} (H^{(5)})^{\mathsf T}$.
Note that $w_2(H^{(5)})^{\mathsf T}=(h_1,h_0,0,\ldots)=(e_1,1,0,\ldots).$
Expanding the Pfaffian along the first row by Lemma~\ref{lem-pf}(ii), we obtain 
\begin{align*}
	\Pf(\widetilde{M})
		&=
			e_1\Pf
			\left(\widetilde{M}_{\{3,4,5,6\},\{3,4,5,6\}}\right)
			-
			\Pf
			\left(\widetilde{M}_{\{2,4,5,6\},\{2,4,5,6\}}\right)\nonumber\\
			&=
			e_1\Pf\left(M_{\{2,3,4,5\}, \{2,3,4,5\}}\right)
			-
			\Pf\left(M_{\{1,3,4,5\},\{1,3,4,5\}}\right).
\end{align*}

By Lemma~\ref{lem:entry}, for $i<j$, 
		\begin{align*}
			M_{ij}
			=\sum_{k\geq3}
			s_{(k-i,k+1-j)}t^{k-3}.
\end{align*}
Consequently, $M_{13}=\sum_{k\geq3}
		s_{(k-1,k-2)}t^{k-3}=\frac{\Yb-e_1}{t}.$
		Similarly, we have
		\begin{align*}
			&M_{14}=\Yc,\quad M_{15}=t\Zf,\quad M_{23}=\frac{\Ya-1}{t}, \quad
			M_{24}=\Yb, \\
			&M_{25}=t\Yc,\quad
			M_{34}=\Ya,\quad
			M_{35}=t\Yb,\quad
			M_{45}=t\Ya.
		\end{align*}
Hence,
		\begin{align*}
			\Pf\left(
			M_{\{2,3,4,5\},\{2,3,4,5\}}
			\right)
			&=M_{23}M_{45}-M_{24}M_{35}+M_{25}M_{34}\\
			&=\frac{\Ya-1}{t}\cdot t\cdot \Ya-\Yb\cdot t\cdot \Yb+t\cdot \Yc\cdot \Ya\\
			&=\Ya^2+t\cdot\left(\Ya \Yc-\Yb^2\right)-\Ya.
		\end{align*}
By~\eqref{eq:base-pf}, $\Pf\left(
			M_{\{2,3,4,5\},\{2,3,4,5\}}
			\right)=\sum_{0\le b\le a}s_{(a,a,b,b)}t^{a+b}-\Ya$.
On the other hand,
		\begin{align*}
			\Pf\left(
			M_{\{1,3,4,5\},\{1,3,4,5\}}
			\right)
			&=M_{13}M_{45}-M_{14}M_{35}+M_{15}M_{34}\\
			&=\frac{\Yb-e_1}{t}\cdot t\cdot \Ya-\Yc\cdot t\cdot \Yb+t\cdot\Zf\cdot \Ya\\
			&=\Ya \Yb-e_1\Ya+t\cdot\left(\Ya\Zf-\Yb\Yc\right).
		\end{align*}
Therefore,
		\begin{align*}
			\Pf(\widetilde M)=e_1 \sum_{0\le b\le a}s_{(a,a,b,b)}\,t^{a+b}-\Ya\Yb+t\cdot\left(\Yb\Yc-\Ya\Zf\right).
		\end{align*}
By Corollary~\ref{thm:augmented}, 
$$
\Pf(\widetilde M)=\sum_{0\le b\le a}
s_{(a+1,a+1,b+1,b+1,1)}
\,t^{a+b+2}.$$
Comparing the two expressions of $\Pf(\widetilde M)$ proves \eqref{eq:base-pf2}.
\end{proof}
\begin{proof}[Proof of \eqref{jr3}]
Let $\mathbf v=(v_1,v_2,\ldots)$ with $v_1=v_{2}=v_3=0$ and $v_i=t^{i-4}$ for $i\ge 4.$
Let $\Omega^{(4)}=\Omega^{(4)}(\mathbf v)$.
Take $n=2,\, p=4$ in \eqref{eq:augmented-pf} and let $\widetilde M=\widetilde H^{(5)}\widetilde\Omega^{(4)}
\bigl(\widetilde H^{(5)}\bigr)^{\mathsf T}$.
Then 
\begin{align*}
\widetilde M&=\begin{bmatrix}
1&0\\
0&H^{(5)}
\end{bmatrix}
\begin{bmatrix}
0&w_{3}\\
-w_{3}^{\mathsf T}&\Omega^{(4)}
\end{bmatrix}
\begin{bmatrix}
1&0\\
0&(H^{(5)})^{\mathsf T}
\end{bmatrix}\\
&=\begin{bmatrix}
			0&w_3(H^{(5)})^{\mathsf T}\\
			-H^{(5)} w_3^{\mathsf T}& H^{(5)}\Omega^{(4)} (H^{(5)})^{\mathsf T}
            \end{bmatrix}
            =\begin{bmatrix}
			0&w_3(H^{(5)})^{\mathsf T}\\
			-H^{(5)} w_3^{\mathsf T}& M
		\end{bmatrix},
\end{align*}
where $M=H^{(5)}\Omega^{(4)} (H^{(5)})^{\mathsf T}$.
Note that $w_3(H^{(5)})^{\mathsf T}=(h_2,h_1,h_0,0,0)=(h_2,e_1,1,0,0).$
Expanding the Pfaffian along the first row by Lemma~\ref{lem-pf}(ii),  we obtain 
		\begin{align*}
			\Pf(\widetilde M)
			&=
			h_2\Pf
			\left(\widetilde M_{\{3,4,5,6\},\{3,4,5,6\}}\right)
			-
			e_1\Pf
			\left(\widetilde M_{\{2,4,5,6\},\{2,4,5,6\}}\right)
			+\Pf
			\left(\widetilde M_{\{2,3,5,6\},\{2,3,5,6\}}\right)\\
			&=
			h_2\Pf\left(M_{\{2,3,4,5\}, \{2,3,4,5\}}\right)
			-
			e_1\Pf\left(M_{\{1,3,4,5\},\{1,3,4,5\}}\right)
            +\Pf\left(M_{\{1,2,4,5\},\{1,2,4,5\}}\right).
		\end{align*}
        
By Lemma~\ref{lem:entry}, for $i<j$, we have
		\begin{align*}
			M_{ij}
			=\sum_{k\geq4}
			s_{(k-i,k+1-j)}t^{k-4}.
		\end{align*}
		Since
		$
		\Ya=1+e_2t+(e_2^2-e_1e_3)t^2+\sum_{k\ge3}s_{(k,k)}t^k,
		$
		we have 
		$$
		M_{12}=\sum_{k\geq4}
		s_{(k-1,k-1)}t^{k-4}=\frac{\Ya-1-e_2t-(e_2^2-e_1e_3)t^2}{t^3},
		$$
		Similarly, 
		\begin{align*}
			&M_{13}=\frac{\Yb-e_1-(e_1e_2-e_3)t}{t^2},\quad
			M_{14}=\frac{\Yc-h_2}{t},\quad
			M_{15}=\Zf,\\
			&M_{23}=\frac{\Ya-1-e_2t}{t^2},\quad
			M_{24}=\frac{\Yb-e_1}{t},\quad
			M_{25}=\Yc,\\
			&M_{34}=\frac{\Ya-1}{t},\quad
			M_{35}=\Yb,\quad
			M_{45}=\Ya.
		\end{align*}
We claim that 
\begin{align}
&h_2\Pf\left(M_{\{2,3,4,5\}, \{2,3,4,5\}}\right)
-e_1\Pf\left(M_{\{1,3,4,5\},\{1,3,4,5\}}\right)
+\Pf\left(M_{\{1,2,4,5\},\{1,2,4,5\}}\right)\nonumber\\
=&\frac{\Ya\Zg-\Zc^2+S_{(k,k,1)}S_{(k,k,1,1,1)}}{t}.
\label{eq:base-pf3}
\end{align}
For readability, we defer the proof of \eqref{eq:base-pf3} to Appendix B.
Therefore, 
		\begin{equation*}
			\Pf(\widetilde M)=\frac{\Ya\Zg-\Zc^2+S_{(k,k,1)}S_{(k,k,1,1,1)}}{t}.
		\end{equation*}
By Corollary~\ref{thm:augmented}, 
$$
\Pf(\widetilde M)=\sum_{0\le b \le a}
s_{(a+2,a+2,b+2,b+2,2)}
\,t^{a+b+2}.$$
Comparing the two expressions of $\Pf(\widetilde M)$ proves \eqref{jr3}.
\end{proof}

\subsection{Vandermonde--Pfaffian Schur expansions}
In this subsection, we show that Corollary~\ref{thm:PfM} can be obtained by the earlier Schur--Pfaffian formulas of Ishikawa, Okada and Wakayama~\cite{IW96}.
They combined the Pfaffian formula~\eqref{pffor} with Vandermonde matrices.
For an arbitrary
skew-symmetric matrix
$
A=[A_{ij}]_{i,j\geq 1}
$
and a Vandermonde matrix $V=[x_i^{j-1}]_{1\le i\le 2n,\, j\ge 1}$,
the formula in~\cite[Theorem~2.3]{IW96} states that
\begin{equation}\label{IW2}
\sum_{\ell(\lambda)\leq 2n}
\operatorname{Pf}\bigl(A_{J(\lambda),J(\lambda)}\bigr)
s_{\lambda}
=
\frac{1}{\prod_{1\le r<s\le 2n}(x_s-x_r)}
\operatorname{Pf}\left(
VAV^{\mathsf T}
\right),
\end{equation}
where $\la_1\ge \la_2\ge \dots\ge \la_{2n}\ge 0$ and
\begin{equation*}
J(\lambda)
=\{\la_{2n}+1,\la_{2n-1}+2,\ldots,\la_2+2n-1,\la_1+2n\}_<.
\end{equation*}
Take
$
A=\Omega(\mathbf v)
$ for 
$\mathbf v=(1,t,t^2,\ldots).
$
Then $A_{i,i+1}=t^{i-1}$ and $A_{i+1,i}=-t^{i-1},$ with $i\ge 1$
and all other entries equal to $0$.
By the definition of the Pfaffian,
$\Pf\bigl(A_{J(\lambda),J(\lambda)}\bigr)\ne 0$ if and only if
$$
(\la_{2n}+1)+1=\la_{2n-1}+2,\,(\la_{2n-2}+3)+1=\la_{2n-3}+4,\,(\la_{2n-4}+5)+1=\la_{2n-5}+6,\ldots.
$$
Equivalently, $\la_{2n-2k+2}=\la_{2n-2k+1}$ for all $1\le k\le n$.
In this case, 
$$
\Pf\bigl(A_{J(\lambda),J(\lambda)}\bigr)=t^{\sum_{k=1}^n \left((\lambda_{2n-2k+2}+2k-1)-1\right)}
=t^{n(n-1)+\sum_{k=1}^n\la_{2n-2k+2}}.
$$
Let $a_k=\la_{2n-2k+2}.$ Then $0\le a_1\le a_2\le \dots\le a_n$
and the left-hand side of~\eqref{IW2} equals
\begin{equation}\label{eq:IW1}
\sum_{0\leq a_1\leq\cdots\leq a_n}
s_{(a^2_n,a^2_{n-1},\ldots,a^2_1)}\,
t^{n(n-1)+\sum_{k=1}^n a_k}.
\end{equation}

Define the $2n\times 2n$ matrix
$
C=[C_{ij}]_{1\le i,j\le 2n}$
with 
$$
C_{ij}=\frac{x_j^{2n-i}}{\prod_{\substack{1\le s\le 2n\\ s\ne j}}
(x_j-x_s)},
$$
and define $V=[x_i^{j-1}]_{1\le i\le 2n, \,j\ge 1}$. 
Recall that $H^{(2n)}=\bigl[h_{j-i}\bigr]_{1\le i\le 2n,\, j\ge1}$.
The following result gives the relation among $H^{(2n)}$, $C$, and $V$.
\begin{proposition}
After specializing $X$ to the finite alphabet $\{x_1,x_2,\ldots,x_{2n}\}$, we have
	\begin{equation*}
	H^{(2n)}=CV.
	\end{equation*}
	Moreover,
\begin{equation*}
\det C=\frac{1}{\prod_{1\le r<s\le 2n}(x_s-x_r)}.
\end{equation*}
\end{proposition}
\begin{proof}
The proof follows from the classical identity of Sylvester (see~\cite[(1.1)]{Bha99})
\begin{equation*}
	h_m
	=
	\sum_{r=1}^{2n}
	\frac{x_r^{m+2n-1}}
	{\displaystyle\prod_{\substack{1\le s\le 2n\\ s\ne r}}
		(x_r-x_s)},
\end{equation*}
for $m\ge0$, together with some elementary linear-algebraic manipulations. 
We will omit the proof here and leave it to the reader.
\end{proof}

For any $2n\times 2n$ matrix $H$ and any $2n\times 2n$ skew-symmetric matrix $\Omega$, Lemma~\ref{pfcomp} reduces to
$
\Pf(H\Omega H^{\mathsf T})
=
\det(H)\Pf(\Omega)$.
The right-hand side of~\eqref{IW2} equals
\begin{equation}\label{eq:IW2}
\frac{1}{\prod_{1\le r<s\le 2n}(x_s-x_r)}
\Pf\left(
VAV^{\mathsf T}
\right)=\det(C)\Pf(VAV^{\mathsf T})=\Pf\left(
	H^{(2n)}A(H^{(2n)})^{\mathsf T}
	\right).
\end{equation}
Combining~\eqref{eq:IW1} and~\eqref{eq:IW2} yields the Schur expansion that appears in Corollary~\ref{thm:PfM}.

\section{Refined Baxter numbers and $q$-refined Baxter numbers}
In this section, we realize the refined and $q$-refined Baxter
numbers as principal specializations.
Recall that \begin{equation*}
	B(n,k)=\frac{\binom{n}{k}\binom{n+1}{k}\binom{n+2}{k}}
	{\binom{k+1}{k}\binom{k+2}{k}}, \quad 0\le k\le n.
\end{equation*}
Motivated by the definition of shifted $q$-binomial coefficients and $q$-Narayana numbers~\cite{Bra04,FH85}, we define the $q$-refined Baxter numbers as 
\begin{align*}
	B_q(n,k)=q^{\frac{3k(k-1)}{2}} \frac{\begin{bmatrix} n \\ k \end{bmatrix}_q \begin{bmatrix} n + 1 \\ k \end{bmatrix}_q \begin{bmatrix} n + 2 \\ k \end{bmatrix}_q}{\begin{bmatrix} k + 1 \\ k \end{bmatrix}_q \begin{bmatrix} k + 2 \\ k \end{bmatrix}_q}, 
\end{align*}
where
\begin{equation*}
	[k]_q \coloneqq \frac{1-q^k}{1-q}, \qquad [k]_q!\coloneqq [1]_q[2]_q\cdots[k]_q,\qquad \begin{bmatrix} n \\ k \end{bmatrix}_q\coloneqq \frac{[n]_q!}{[k]_q![n-k]_q!}.
\end{equation*}
For $n\geq 0$ and $k\in\mathbb Z$, we adopt the convention
\[
B(n,k)=B_q(n,k)=0
\quad\text{unless}\quad 0\leq k\leq n.
\]
We remark that Cigler~\cite{Cig21} defined a different $q$-analogue of the refined Baxter numbers.

For a Schur function $s_\lambda$,
its \emph{principal specialization} is defined by 
$$
\ps_n(s_\lambda)=s_\lambda(1,q,\ldots,q^{n-1})\quad \textrm{and}\quad \ps^1_n(s_\lambda)=s_\lambda(1^n).
$$ In particular,
\[
\operatorname{ps}_0^1(s_\lambda)
=
\begin{cases}
1, & \lambda=\varnothing,\\
0, & \lambda\neq\varnothing.
\end{cases}
\]
The following \emph{hook-content formula} was first stated explicitly by Stanley~\cite{Sta71}. 

\begin{lemma}{\rm \cite[Theorem 7.21.2, Corollary 7.21.4]{Sta24}}
	\label{lem:hook-content}
	Let $\lambda$ be a partition with at most $n$ parts.
Then
	\begin{equation*}
		\ps_n(s_\lambda)
		=q^{\sum_{i\geq 1}(i-1)\lambda_i}
		\prod_{(i,j)\in\lambda}
		\frac{1-q^{n+j-i}}{1-q^{h(i,j)}}
	\end{equation*}
and
	\begin{equation*}
		\ps_n^1(s_\lambda)
		=\prod_{(i,j)\in\lambda}
		\frac{n+j-i}{h(i,j)}, 
	\end{equation*}
where $h(i,j)$ is the hook length of $\lambda$ at $(i,j)$.
\end{lemma}

The following result is motivated by Cigler~\cite{Cig21}.

\begin{lemma}\label{BqSchur}
	For $n\ge k\ge 0$, we have $B_q(n,k)=\ps_n\left(s_{(3^k)}\right)$.
\end{lemma}

\begin{proof}
	For the partition $\lambda = (3^k)$, it is easy to check that
	\begin{equation*}
		\sum_{i\ge1} (i-1)\la_i=\frac{3k(k-1)}{2} \quad \text{and} \quad h(i,j)=k-i-j+4.
	\end{equation*}
	By Lemma~\ref{lem:hook-content}, we deduce that
	\begin{align}
		\ps_n\left(s_{(3^k)}\right)&=q^{\frac{3k(k-1)}{2}}\prod_{(i,j)\in\lambda}
		\frac{1-q^{n+j-i}}{1-q^{h(i,j)}}\nonumber\\
		& = q^{\frac{3k(k-1)}{2}}\prod_{i=1}^k \prod_{j=1}^3 \frac{[n - i + j]_q}{[k- i - j + 4]_q}\label{eq-expre}\\
		&=q^{\frac{3k(k-1)}{2}}\prod_{i=1}^k \frac{[n - i + 1]_q[n - i + 2]_q[n - i + 3]_q}{[k - i + 1]_q[k - i + 2]_q[k - i + 3]_q}.\nonumber
	\end{align}
	Since
	\begin{equation*}
		\prod_{i=1}^k [n - i + 1]_q  [n - i + 2]_q  [n - i + 3]_q = \frac{[n]_q!}{[n - k]_q!} \cdot \frac{[n + 1]_q!}{[n - k + 1]_q!} \cdot \frac{[n + 2]_q!}{[n - k + 2]_q!}
	\end{equation*}
	and
	\begin{equation*}
		\prod_{i=1}^k [k - i + 1]_q  [k - i + 2]_q  [k - i + 3]_q = \frac{[k]_q!}{[0]_q!} \cdot \frac{[k + 1]_q!}{[1]_q!} \cdot \frac{[k + 2]_q!}{[2]_q!} = \frac{[k]_q![k + 1]_q![k + 2]_q!}{[2]_q!},
	\end{equation*}
	so we have 
	\begin{align*}
		\ps_n\left(s_{(3^k)}\right)&=q^{\frac{3k(k-1)}{2}}
		\frac{[n]_q![n + 1]_q![n + 2]_q! \cdot [2]_q!}{[k]_q![k + 1]_q![k + 2]_q! \cdot [n - k]_q![n - k + 1]_q![n - k + 2]_q!}\\
		&=q^{\frac{3k(k-1)}{2}}\frac{\begin{bmatrix} n \\ k \end{bmatrix}_q \begin{bmatrix} n + 1 \\ k \end{bmatrix}_q \begin{bmatrix} n + 2 \\ k \end{bmatrix}_q}{\begin{bmatrix} k + 1 \\ k \end{bmatrix}_q \begin{bmatrix} k + 2 \\ k \end{bmatrix}_q}=B_q(n,k).
	\end{align*}
	This completes the proof.
\end{proof}

The following result is immediate.
\begin{lemma}\label{coeff}
	For all $n,k\in \mathbb{N}$,
	we have $B(n,k)=\ps^1_n\left(s_{(3^k)}\right)$.
\end{lemma}

We next use~\eqref{multi} to derive two identities for the
principal specializations of the Schur functions $s_{\lambda}$ indexed by three-column shapes,
together with the fact that a Schur function is invariant under a $180^{\circ}$ rotation of its diagram (see \cite[Exercise 7.56(a)]{Sta24}).

\begin{lemma}\label{lem-for-c1}
For $m\ge n\ge 1$ and $k\ge 0$, we have 
\begin{equation*}
	\ps_m^1\left(s_{(3^k)}\right)
	=
	\sum_{0\le a\le b\le c\le m-n}
	\ps_n^1\!\left(s_{(3^{k-c},2^{c-b},1^{b-a})}\right)
	\ps_{m-n}^1\!\left(s_{(3^{a},2^{b-a},1^{c-b})}\right).
\end{equation*}
\end{lemma}
\begin{proof}
Let $\lambda=3^k$, $\mu=\emptyset$.
Set $\nu=(3^a,2^{b-a},1^{c-b})$ with $0\le a\le b\le c\le m-n$. Then by~\eqref{multi}, we have
\begin{align*}
&s_{(3^k)}(1,q,\ldots,q^{m-1})\\
=&\sum_{0\le a\le b\le c\le m-n}
s_{(3^{k-c},2^{c-b},1^{b-a})}(1,q,\ldots,q^{n-1})\,
s_{(3^{a},2^{b-a},1^{c-b})}(q^n,\ldots,q^{m-1}).
\end{align*}
We complete the proof by taking $q=1$.
\end{proof}

\begin{lemma}\label{lem-for-c2}
Let $k\ge0$ and $n\ge1$.
For any $a<0$ or $b<0$, set
$s_{(2^a,1^b)}=0$ by convention. Then we have
\begin{align*}
	\ps_n\left(s_{(3^{k})}\right)
	=
	&
	\ps_{n-1}\left(s_{(3^{k})}\right)
	+q^{\,n-1}
	\ps_{n-1}\left(s_{(3^{k-1},2)}\right)
	\\
	&
	+q^{2(n-1)}
	\ps_{n-1}\left(s_{(3^{k-1},1)}\right)
	+q^{3(n-1)}
	\ps_{n-1}\left(s_{(3^{k-1})}\right),
\end{align*}
Moreover, 
\begin{align*}
	\ps_n^{1}\!\left(s_{(3^{k})}\right)
	=
	\ps_{n-1}^{1}\Big(
	s_{(3^{k})}
	+s_{(3^{k-1},2)}
	+s_{(3^{k-1},1)}
	+s_{(3^{k-1})}
	\Big).
\end{align*}
\end{lemma}
\begin{proof}
Let $\lambda=(3^k)$, $\mu=\emptyset$ and let $Y=\{q^{n-1}\}$ in~\eqref{multi}.
Then $\nu\in\{\emptyset,(1),(2), (3)\}$ and 
\begin{align*}
s_{(3^k)}(1,q,\ldots,q^{n-1})=
&s_{(3^k)}(1,q,\ldots,q^{n-2})+s_{(3^{k-1},2)}(1,q,\ldots,q^{n-2}) s_{(1)}(q^{n-1})\\
+&s_{(3^{k-1},1)}(1,q,\ldots,q^{n-2})s_{(2)}(q^{n-1})+s_{(3^{k-1})}(1,q,\ldots,q^{n-2}) s_{(3)}(q^{n-1}).
\end{align*}
Then the first equation follows. 
We obtain the second equation by taking $q=1$.
\end{proof}

\section{$q$-Log-convexity}\label{sec:q-log-convexity}

Let
\begin{equation*}
B_n(x)=\sum_{k=0}^{n}B(n,k)x^k.
\end{equation*}
Thus $B_n(1)=B_n$. 
We call $B_n(x)$ the $n$-th \emph{Baxter polynomial}. 
The main objective of this section is to prove the $q$-log-convexity of the Baxter polynomials, i.e.,
$$
\left[x^r\right]\left(B_{n-1}(x)B_{n+1}(x)-B_n(x)^2\right)\ge 0
$$ 
for all $0\le r\le 2n$, where $\left[x^r\right] f(x)$ is the coefficient of $x^r$ in $f(x)$.
Since $q$ is already used in the definition of $B_q(n,k)$,
we use $x$ as the polynomial variable to avoid confusion.

\begin{theorem}\label{q-lcx}
The Baxter polynomials $B_n(x)$ form a $q$-log-convex sequence. 
\end{theorem}

Taking $x=1$, it is immediate that the sequence of the Baxter numbers 
$(B_n)_{n\ge0}$ forms a log-convex sequence,
which was first proved in~\cite[Proposition 6.2]{MS26}.

For $n\ge 1$ and $r\ge 0$, define 
$$
C_1(r)=\left[x^r\right]\left(B_{n+1}(x)B_{n-1}(x)\right)\quad \textrm{and}
\quad C_2(r)=\left[x^r\right]\left(B_n(x)^2\right).
$$
Then, by Lemma~\ref{coeff}, we have
$$
C_1(r)=
\sum_{k=0}^{r}
\ps_{n+1}^{1}\!\left(s_{(3^k)}\right)
\ps_{n-1}^{1}\!\left(s_{(3^{\,r-k})}\right)\quad \textrm{and} \quad 
C_2(r)=
\sum_{k=0}^{r}
\ps_{n}^{1}\!\left(s_{(3^k)}\right)
\ps_{n}^{1}\!\left(s_{(3^{\,r-k})}\right).
$$
For $r\geq0$ and $0\leq k\leq r$, define
\begin{align*}
	F(r,k)=
	&\Big(
	s_{(3^{k-2},2^2)}
	+2\,s_{(3^{k-2},2,1)}
	+s_{(3^{k-1},1)}
	+s_{(3^{k-2},1^2)}
	+2\,s_{(3^{k-1})}
	+s_{(3^{k-2},2)}
	\Big) \cdot s_{(3^{r-k})}\\\nonumber
	&-
	\Big(
	s_{(3^{k-1},2)}+s_{(3^{k-1},1)}
	\Big)\cdot
	\Big(
	s_{(3^{r-k-1},2)}+s_{(3^{r-k-1},1)}
	\Big),
\end{align*}
where $F(0,0)=0$.

\begin{theorem}\label{thm:main}
	For $r\geq1$,  the symmetric function $\sum_{k=0}^{r}F(r,k)$ is Schur positive.
\end{theorem}

The Schur positivity of $\sum_{k=0}^{r}F(r,k)$ plays a key role in the proof of Theorem~\ref{q-lcx}.
We will prove Theorem~\ref{thm:main} at the end of this section.
We now prove the $q$-log-convexity of $(B_n(x))_{n\ge 0}.$

\begin{proof}[Proof of Theorem~\ref{q-lcx}]
The case $n=1$ follows directly from
$$B_0(x)B_2(x)-B_1(x)^2=(1+4x+x^2)-(1+x)^2=2x.$$
Now assume that $n\ge 2$.
It suffices to prove that $C_1(r)-C_2(r)\ge 0$
	for all $r\ge 0$.
We claim that
	\begin{equation}\label{c1-c2}
		C_1(r)-C_2(r)=\ps_{n-1}^1 \left(\sum_{k=0}^{r}F(r,k)\right).
	\end{equation}
	Theorem~\ref{thm:main} gives the Schur positivity of $\sum_{k=0}^{r}F(r,k)$.
	Taking the principal specializations of $\sum_{k=0}^{r}F(r,k)$,
	we obtain $C_1(r)-C_2(r)\ge 0.$
	
	We now give a proof of~\eqref{c1-c2}. 
	The proof is divided into three steps.
	The first two steps give the expansions of $C_1(r)$ and $C_2(r) $.
	We present an explicit expression of $C_1(r)-C_2(r)$ in the last step.

    \medskip
    \noindent
	\textbf{Expansion of $C_1(r)$.}\\
	Substituting $m=n+1$ and $n=n-1$ in Lemma~\ref{lem-for-c1},  it follows that
	$$
	\ps_{n+1}^{1}\!\left(s_{(3^k)}\right)
	=
	\sum_{0\le a\le b\le c\le2}
	\ps_{n-1}^{1}\!\left(s_{(3^{k-c},2^{c-b},1^{b-a})}\right)
	\ps_{2}^{1}\!\left(s_{(3^{a},2^{b-a},1^{c-b})}\right).
	$$
	Using the fact that 
	$$
	\ps_{n-1}^{1}(s_{\lambda_1})\,
	\ps_{n-1}^{1}(s_{\lambda_2})
	=
	\ps_{n-1}^{1}(s_{\lambda_1}s_{\lambda_2}),
	$$
	we obtain 
	\begin{equation*}
		C_1(r)
		=
		\sum_{k=0}^{r}
		\sum_{0\le a\le b\le c\le2}
		\ps_{2}^{1}\!\left(s_{(3^{a},2^{b-a},1^{c-b})}\right)
		\ps_{n-1}^{1}
		\!\left(
		s_{(3^{k-c},2^{c-b},1^{b-a})}
		\, s_{(3^{r-k})}
		\right).
	\end{equation*}
	Substituting $\mu=\emptyset$ and applying~\eqref{multi} with the one-variable
     alphabets $X = \{1\}$ and  $Y =\{1\}$ gives
	\begin{align*}
		\ps_{2}^{1}
		\left(
		s_{(3^a,2^{b-a},1^{c-b})}
		\right)
		={}&
		\ps_{1}^{1}
		\left(
		s_{(3^a,2^{b-a},1^{c-b})}
		\right)
		+
		\ps_{1}^{1}
		\left(
		s_{(3^a,2^{b-a},1^{c-b-1})}
		\right)
		\\
		&
		+
		\ps_{1}^{1}
		\left(
		s_{(3^a,2^{b-a-1},1^{c-b+1})}
		\right)
		+
		\ps_{1}^{1}
		\left(
		s_{(3^{a-1},2^{b-a+1},1^{c-b})}
		\right)
		\\
		&
		+
		\ps_{1}^{1}
		\left(
		s_{(3^a,2^{b-a-1},1^{c-b})}
		\right)
		+
		\ps_{1}^{1}
		\left(
		s_{(3^{a-1},2^{b-a+1},1^{c-b-1})}
		\right)
		\\
		&
		+
		\ps_{1}^{1}
		\left(
		s_{(3^{a-1},2^{b-a},1^{c-b+1})}
		\right)
		+
		\ps_{1}^{1}
		\left(
		s_{(3^{a-1},2^{b-a},1^{c-b})}
		\right).
	\end{align*}
	
	Hence, $C_1(r)$ can be decomposed into eight sums:
	\begin{align*}
		A_1(r)=&\sum_{k=0}^{r}\sum_{0\le a\le b\le c\le2}
		\ps_{1}^{1}
		\left(
		s_{(3^a,2^{b-a},1^{c-b})}
		\right)
		\ps_{n-1}^{1}
		\!\left(
		s_{(3^{k-c},2^{c-b},1^{b-a})}
		\, s_{(3^{r-k})}
		\right),\\
		A_2(r)=&\sum_{k=0}^{r}\sum_{0\le a\le b\le c\le2}
		\ps_{1}^{1}
		\left(
		s_{(3^a,2^{b-a},1^{c-b-1})}
		\right)
		\ps_{n-1}^{1}
		\!\left(
		s_{(3^{k-c},2^{c-b},1^{b-a})}
		\, s_{(3^{r-k})}
		\right),\\
		A_3(r)=&\sum_{k=0}^{r}\sum_{0\le a\le b\le c\le2}
		\ps_{1}^{1}
		\left(
		s_{(3^a,2^{b-a-1},1^{c-b+1})}
		\right)
		\ps_{n-1}^{1}
		\!\left(
		s_{(3^{k-c},2^{c-b},1^{b-a})}
		\, s_{(3^{r-k})}
		\right),\\
		A_4(r)=&\sum_{k=0}^{r}\sum_{0\le a\le b\le c\le2}
		\ps_{1}^{1}
		\left(
		s_{(3^{a-1},2^{b-a+1},1^{c-b})}
		\right)
		\ps_{n-1}^{1}
		\!\left(
		s_{(3^{k-c},2^{c-b},1^{b-a})}
		\, s_{(3^{r-k})}
		\right),\\
		A_5(r)=&\sum_{k=0}^{r}\sum_{0\le a\le b\le c\le2}
		\ps_{1}^{1}
		\left(
		s_{(3^a,2^{b-a-1},1^{c-b})}
		\right)
		\ps_{n-1}^{1}
		\!\left(
		s_{(3^{k-c},2^{c-b},1^{b-a})}
		\, s_{(3^{r-k})}
		\right),\\
		A_6(r)=&\sum_{k=0}^{r}\sum_{0\le a\le b\le c\le2}
		\ps_{1}^{1}
		\left(
		s_{(3^{a-1},2^{b-a+1},1^{c-b-1})}
		\right)
		\ps_{n-1}^{1}
		\!\left(
		s_{(3^{k-c},2^{c-b},1^{b-a})}
		\, s_{(3^{r-k})}
		\right),\\
		A_7(r)=&\sum_{k=0}^{r}\sum_{0\le a\le b\le c\le2}
		\ps_{1}^{1}
		\left(
		s_{(3^{a-1},2^{b-a},1^{c-b+1})}
		\right)
		\ps_{n-1}^{1}
		\!\left(
		s_{(3^{k-c},2^{c-b},1^{b-a})}
		\, s_{(3^{r-k})}
		\right),\\
		A_8(r)=&\sum_{k=0}^{r}\sum_{0\le a\le b\le c\le2}
		\ps_{1}^{1}
		\left(
		s_{(3^{a-1},2^{b-a},1^{c-b})}
		\right)
		\ps_{n-1}^{1}
		\!\left(
		s_{(3^{k-c},2^{c-b},1^{b-a})}
		\, s_{(3^{r-k})}
		\right).
	\end{align*}
	
	\medskip
    \noindent
	\textbf{Expansion of $C_2(r)$.}\\
	Substituting $m=n$ and $n=n-1$ in Lemma~\ref{lem-for-c1},  it follows that
	\begin{equation*}
		\ps_n^{1}\!\left(s_{(3^k)}\right)
		=
		\sum_{0\le a\le b\le c\le1}
		\ps_{n-1}^{1}
		\!\left(
		s_{(3^{k-c},2^{c-b},1^{b-a})}
		\right)
		\ps_{1}^{1}
		\!\left(
		s_{(3^{a},2^{b-a},1^{c-b})}
		\right).
	\end{equation*}
	Substituting $k=r-k$ in Lemma~\ref{lem-for-c2},  it follows that 
	$$
	\ps_n^{1}\!\left(s_{(3^{r-k})}\right)
	=
	\ps_{n-1}^{1}\Big(
	s_{(3^{r-k})}
	+s_{(3^{r-k-1},2)}
	+s_{(3^{r-k-1},1)}
	+s_{(3^{r-k-1})}
	\Big).
	$$
	Thus, we obtain 
	\begin{align*}
		C_2(r)
		&=
		\sum_{k=0}^{r}
		\sum_{0\le a\le b\le c\le1}
		\ps_{1}^{1}\!\left(s_{(3^{a},2^{b-a},1^{c-b})}\right) \cdot
		\\
		&
		\ps_{n-1}^{1}
		\!\Big(
		s_{(3^{k-c},2^{c-b},1^{b-a})}
		\big(
		s_{(3^{r-k})}
		+s_{(3^{r-k-1},2)}
		+s_{(3^{r-k-1},1)}
		+s_{(3^{r-k-1})}
		\big)
		\Big).
	\end{align*}
 So $C_2(r)$ can be decomposed into four sums:
	\begin{align*}
		D_1(r)&=\sum_{k=0}^{r}
		\sum_{0\le a\le b\le c\le1}
		\ps_{1}^{1}\!\left(s_{(3^{a},2^{b-a},1^{c-b})}\right)
		\ps_{n-1}^{1} \Big(
		s_{(3^{k-c},2^{c-b},1^{b-a})}
		\,
		s_{(3^{r-k})}
		\Big),\\
		D_2(r)&=\sum_{k=0}^{r}
		\sum_{0\le a\le b\le c\le1}
		\ps_{1}^{1}\!\left(s_{(3^{a},2^{b-a},1^{c-b})}\right)
		\ps_{n-1}^{1} \Big(
		s_{(3^{k-c},2^{c-b},1^{b-a})}
		\,
		s_{(3^{r-k-1},2)}
		\Big),\\
		D_3(r)&=\sum_{k=0}^{r}
		\sum_{0\le a\le b\le c\le1}
		\ps_{1}^{1}\!\left(s_{(3^{a},2^{b-a},1^{c-b})}\right)
		\ps_{n-1}^{1} \Big(
		s_{(3^{k-c},2^{c-b},1^{b-a})}
		\,
		s_{(3^{r-k-1},1)}
		\Big),\\
		D_4(r)&=\sum_{k=0}^{r}
		\sum_{0\le a\le b\le c\le1}
		\ps_{1}^{1}\!\left(s_{(3^{a},2^{b-a},1^{c-b})}\right)
		\ps_{n-1}^{1} \Big(
		s_{(3^{k-c},2^{c-b},1^{b-a})}
		\,
		s_{(3^{r-k-1})}
		\Big).
	\end{align*}
	Hence, $$C_1(r)-C_2(r)=\sum_{i=1}^8A_i(r)-\sum_{i=1}^4D_i(r).$$
	
	\medskip
    \noindent
	\textbf{Simplification of $C_1(r)-C_2(r)$}\\
	We use $A_i$ (resp. $D_i$) to denote $A_i(r)$ (resp. $D_i(r)$) for brevity.
    All additional boundary terms vanish by
    our convention that a Schur function indexed by an invalid partition
    is zero.
	We divide the simplification into five cases: (i) $A_1=D_1$; (ii) $A_8=D_4$;
	\begin{align*}
		&\textrm{(iii)} \quad A_2-D_2
		=\sum_{k=0}^{r}\ps_{n-1}^1\Big(
		s_{(3^{k-2},2^2)} \, s_{(3^{r-k})}
		+s_{(3^{k-2},2,1)}\, s_{(3^{r-k})}\\
		&\qquad \qquad \qquad \qquad -s_{(3^{k-1},2)}\,s_{(3^{r-k-1},2)}
		-s_{(3^{k-1},1)}\, s_{(3^{r-k-1},2)}
		\Big);\\
		&\textrm{(iv)} \quad A_3+A_5-D_3=\sum_{k=0}^{r}\ps_{n-1}^1\Big(
		s_{(3^{k-1},1)}\, s_{(3^{r-k})}
		+s_{(3^{k-2},2,1)}\, s_{(3^{r-k})}
		+s_{(3^{k-2},1^2)}\, s_{(3^{r-k})}\\
		&\qquad \qquad \qquad \qquad \qquad -s_{(3^{k-1},2)}\, s_{(3^{r-k-1},1)}
		-s_{(3^{k-1},1)}\, s_{(3^{r-k-1},1)}
		\Big);\\
		&\textrm{(v)} \quad A_4+A_6+A_7
		=\sum_{k=0}^{r}
		\ps_{n-1}^1
		\Big(
		2\,s_{(3^{k-1})}\, s_{(3^{r-k})}
		+s_{(3^{k-2},2)}\, s_{(3^{r-k})}\Big).
	\end{align*}
	Summarizing all five cases, and by the linearity of $\ps_{n-1}^1$,
	we derive that $C_1(r)-C_2(r)$ equals
	\begin{align*}
		& \ps_{n-1}^1
		\sum_{k=0}^{r}\Bigg[
		\Big(
		s_{(3^{k-2},2^2)}
		+2\,s_{(3^{k-2},2,1)}
		+s_{(3^{k-1},1)}
		+s_{(3^{k-2},1^2)}
		+2\,s_{(3^{k-1})}
		+s_{(3^{k-2},2)}
		\Big)\cdot s_{(3^{r-k})}\\
		&\qquad-
		\Big(
		s_{(3^{k-1},2)}+s_{(3^{k-1},1)}
		\Big)\cdot
		\Big(
		s_{(3^{r-k-1},2)}+s_{(3^{r-k-1},1)}
		\Big)
		\Bigg].
	\end{align*}
	This proves~\eqref{c1-c2}. 
	We give the details for (i) and (ii). For readability, the remaining cases (iii)--(v) are deferred to Appendix A.
	Recall that 
	$$
	\ps_1^1(s_{\lambda}) = \begin{cases}
		1, & \textrm{if} \ \lambda=\emptyset \ \textrm{or}\  \ell(\lambda)=1, \\
		0, & \textrm{otherwise}.
	\end{cases}
	$$
	\begin{itemize}
		\item[(i)] Proof that $A_1=D_1$. 
		Clearly, $\ps_1^1(s_{(3^a,2^{b-a},1^{c-b})})=0$ for $c=2,$ since
		$\ell(3^a,2^{b-a},1^{c-b})=c$. Then
		\begin{align*}
			A_1=&\sum_{k=0}^{r}\sum_{0\le a\le b\le c\le2}
			\ps_{1}^{1}
			\left(
			s_{(3^a,2^{b-a},1^{c-b})}
			\right)
			\ps_{n-1}^{1}
			\!\left(
			s_{(3^{k-c},2^{c-b},1^{b-a})}
			\, s_{(3^{r-k})}
			\right)\\
			=&\sum_{k=0}^{r}\sum_{0\le a\le b\le c\le1}
			\ps_{1}^{1}
			\left(
			s_{(3^a,2^{b-a},1^{c-b})}
			\right)
			\ps_{n-1}^{1}
			\!\left(
			s_{(3^{k-c},2^{c-b},1^{b-a})}
			\, s_{(3^{r-k})}
			\right)=D_1.
		\end{align*}
		\item[(ii)]  Proof that $A_8=D_4$.
		Clearly, $\ps_1^1(s_{(3^{a-1},2^{b-a},1^{c-b})})=0$ for $a<1.$
		Then
		\begin{align*}
			A_8=&\sum_{k=0}^{r}\sum_{0\le a\le b\le c\le2}
			\ps_{1}^{1}
			\left(
			s_{(3^{a-1},2^{b-a},1^{c-b})}
			\right)
			\ps_{n-1}^{1}
			\!\left(
			s_{(3^{k-c},2^{c-b},1^{b-a})}
			\, s_{(3^{r-k})}
			\right)\\
			=&\sum_{k=0}^{r}\sum_{1\le a\le b\le c\le2}
			\ps_{1}^{1}
			\left(
			s_{(3^{a-1},2^{b-a},1^{c-b})}
			\right)
			\ps_{n-1}^{1}
			\!\left(
			s_{(3^{k-c},2^{c-b},1^{b-a})}
			\, s_{(3^{r-k})}
			\right)\\
			=&\sum_{k=0}^{r}\sum_{0\le a\le b\le c\le 1}
			\ps_{1}^{1}
			\left(
			s_{(3^{a},2^{b-a},1^{c-b})}
			\right)
			\ps_{n-1}^{1}
			\!\left(
			s_{(3^{k-c-1},2^{c-b},1^{b-a})}
			\, s_{(3^{r-k})}
			\right)\\
			=&\sum_{k=0}^{r-1}\sum_{0\le a\le b\le c\le 1}
			\ps_{1}^{1}
			\left(
			s_{(3^{a},2^{b-a},1^{c-b})}
			\right)
			\ps_{n-1}^{1}
			\!\left(
			s_{(3^{k-c},2^{c-b},1^{b-a})}
			\, s_{(3^{r-k-1})}
			\right)=D_4.
		\end{align*}
	\end{itemize}
\end{proof}
\subsection{Proof of Theorem~\ref{thm:main}}
In this subsection, we prove the Schur positivity of 
$\sum_{k=0}^{r}F(r,k)$,
where
\begin{align*}
	F(r,k)={}&\Bigl(
	s_{(3^{k-2},2^2)}+2\,s_{(3^{k-2},2,1)}+s_{(3^{k-2},1^2)}
	+s_{(3^{k-2},2)}+s_{(3^{k-1},1)}+2\,s_{(3^{k-1})}
	\Bigr)\cdot s_{(3^{r-k})}\\
	&-\Bigl(s_{(3^{k-1},2)}+s_{(3^{k-1},1)}\Bigr)\cdot
	\Bigl(s_{(3^{r-k-1},2)}+s_{(3^{r-k-1},1)}\Bigr).
\end{align*}
We will prove the Schur positivity of the standard involution $\omega\left(\sum_{k=0}^{r}F(r,k)\right)$,
which,  by Lemma~\ref{SPOMEGA}, implies the Schur positivity of $\sum_{k=0}^{r}F(r,k)$.

Note that every term in $\omega(F(r,k))$ is the product of two Schur functions,
and the length of the partition of every Schur function in $\omega(F(r,k))$ is at most $3$.
Hence, by the Littlewood--Richardson rule,
the length of the partition of every term in $\omega(F(r,k))$ is at most $6$.
We first prove the Schur positivity of a part of $\omega\left(\sum_{k=0}^{r}F(r,k)\right)$,
that is, the terms whose corresponding partition has length at most $5$.
Theorems~\ref{thm:base-pf} and~\ref{thm:base-pf22} play a key role in the proof of Schur positivity.
Finally, relating the general terms in $\omega(F(r,k))$ and the terms whose corresponding partition has length at most $5$,
we derive the Schur positivity of $\omega\left(\sum_{k=0}^{r}F(r,k)\right)$, and Theorem~\ref{thm:main} follows.


It is not difficult to verify that
\begin{align*}
	\omega(F(r,k))=&\Big(s_{(k,k,k-2)}+2\,s_{(k,k-1,k-2)}+s_{(k,k-2,k-2)}+s_{(k-1,k-1,k-2)}\\
	&+s_{(k,k-1,k-1)}+2\,s_{(k-1,k-1,k-1)}\Big)\cdot s_{(r-k,r-k,r-k)}\\
	&-\left(s_{(k,k,k-1)}+s_{(k,k-1,k-1)}\right)\cdot\left(s_{(r-k,r-k,r-k-1)}+s_{(r-k,r-k-1,r-k-1)}\right),
\end{align*}
and 
\begin{align}
\omega\left(\sum_{k=0}^{r}F(r,k)\right)=
&\sum_{k=0}^{r}\Big[\left(s_{(k,k,k-2)}+s_{(k,k-1,k-1)}\right)+2\,\left(s_{(k,k-1,k-2)}+s_{(k-1,k-1,k-1)}\right)\nonumber\\
	&\qquad+\left(s_{(k,k-2,k-2)}+s_{(k-1,k-1,k-2)}\right)\Big]\cdot s_{(r-k,r-k,r-k)}\nonumber\\
	&-\sum_{k=0}^{r}\Big(s_{(k,k,k-1)}\, s_{(r-k,r-k,r-k-1)}+2\,s_{(k,k,k-1)} \, s_{(r-k,r-k-1,r-k-1)}\nonumber\\
	&\qquad+s_{(k,k-1,k-1)}\, s_{(r-k,r-k-1,r-k-1)} \Big).\label{omega1}
\end{align}
Let $\omega\left(\sum_{k=0}^{r}F(r,k)\right)=F_1(r)+2\,F_2(r)+F_3(r)$,
where
\begin{align*}
	F_1(r)&=\sum_{k=0}^{r}
	\left[\left(s_{(k,k,k-2)}+s_{(k,k-1,k-1)}\right)\cdot s_{(r-k,r-k,r-k)}-s_{(k,k,k-1)}\cdot s_{(r-k,r-k,r-k-1)}\right],
	\\
	F_2(r)&=\sum_{k=0}^{r}
	\left[\left(s_{(k,k-1,k-2)}+s_{(k-1,k-1,k-1)}\right)\cdot s_{(r-k,r-k,r-k)}-s_{(k,k,k-1)} \cdot s_{(r-k,r-k-1,r-k-1)}\right],
	\\
	F_3(r)&=\sum_{k=0}^{r}
	\left[\left(s_{(k,k-2,k-2)}+s_{(k-1,k-1,k-2)}\right)\cdot s_{(r-k,r-k,r-k)}-s_{(k,k-1,k-1)}\cdot s_{(r-k,r-k-1,r-k-1)}\right].
\end{align*}
We shall prove the Schur positivity of $F_1(r), F_2(r), F_3(r)$,
which implies the Schur positivity of $\omega\left(\sum_{k=0}^{r}F(r,k)\right)$.

\subsubsection{The terms of length at most five}
\label{subsec:<=5}
Clearly, for every factor $s_{\la}$ in $\omega\left(\sum_{k=0}^{r}F(r,k)\right)$, we have $\ell\left({\la}\right)\le 3$.
By the Littlewood--Richardson rule, for every term $s_{\mu} s_{\nu}$, 
the length of each partition in its Littlewood--Richardson expansion is at most $6$.
Suppose that for $i=1,2,3$,
$$
F_i(r)=\sum_{\lambda} a_{\lambda} s_{\lambda}.
$$
Hence, the sum is over $\la$ with $\ell(\la)\le 6.$
Define 
$$
F_i(r)\big|_{\ell\leq5}=\sum_{\substack{\la \\\ell(\la)\le 5}} a_{\lambda} s_{\lambda}.
$$
Then for $X=\{x_1,\ldots,x_5\}$,
\[
F_i(r)\big|_{\ell\le5}(X)=F_i(r)(X).
\]
In this subsection, we present the Schur positivity of $F_i(r)\big|_{\ell\leq5}$ for $i=1,2,3$.
We introduce some notation.
Let $\pi_5=x_1x_2x_3x_4 x_5$.  Suppose that
$\lambda=(\lambda_1,\la_2,\la_3,\la_4,\la_5)\subseteq(a^5)$.  
Define the complement of $\la$ as 
\begin{equation}\label{eq-com}
	\lambda^{\vee,a}=(a-\lambda_5,a-\lambda_4,\dots,a-\lambda_1), 
\end{equation}
and
$
s^{\vee,a}_{\la}:=s_{\la^{\vee,a}}.
$
Note that $s^{\vee,a}_{\la}+s^{\vee,a}_{\tau}=(s_\la+s_\tau)^{\vee,a}$ if $\la,\tau\subseteq(a^5)$.
Define 
$$
f^{\vee,a}= \sum_\la a_\la s^{\vee,a}_\la
$$ 
if the expansion $f=\sum_\la a_\la s_\la$ with all $\la \subseteq(a^5)$.
\begin{lemma}{\rm \cite[p. 457]{Sta24}}
	\label{lem:complement}
	For $\lambda\subseteq(a^5)$ and $X=\{x_1,\ldots,x_5\}$, we have 
	\begin{equation}\label{comple}
		s_\lambda(X)=\pi_5^a
		s^{\vee,a}_{\lambda}(x_1^{-1},\dots,x_5^{-1})=\pi_5^a s^{\vee,a}_{\lambda}(X^{-1}).
	\end{equation}
\end{lemma}

We use $s_\mu s_\nu$ to denote the term in $F_i(r)$, where $s_\mu$ denotes the first factor 
and $s_\nu$ denotes the second.
Clearly, for every term $s_\mu s_\nu$  with $\mu=(\mu_1,\mu_2,\mu_3)$ and $\nu=(\nu_1,\nu_2,\nu_3)$,
we have $\mu_1\le k$ and $\nu_1\le r-k.$ 
Thus the partition of  first factor $\mu \subseteq(k^5)$
and the partition of second factor $\nu\subseteq((r-k)^5)$.
We list the required complements of the first factor:
\begin{align*}
	&s^{\vee,k}_{(k,k,k)}=s_{(k,k)},\quad
	s^{\vee,k}_{(k,k,k-1)}=s_{(k,k,1)},\quad
	s^{\vee,k}_{(k,k-1,k-1)}=s_{(k,k,1,1)},\\
	&s^{\vee,k}_{(k,k,k-2)}=s_{(k,k,2)},\quad
	s^{\vee,k}_{(k,k-1,k-2)}=s_{(k,k,2,1)},\quad
	s^{\vee,k}_{(k,k-2,k-2)}=s_{(k,k,2,2)},\\
	&s^{\vee,k}_{(k-1,k-1,k-2)}=s_{(k,k,2,1,1)},\quad
	s^{\vee,k}_{(k-1,k-1,k-1)}=s_{(k,k,1,1,1)}.
\end{align*}
\begin{lemma}\label{slsm}
	Let $s_\mu s_\nu$ be a product occurring in $F_i(r)$ for $i=1,2,3,$
	where $s_\mu$ is indexed by a partition contained in $(k^5)$
	and $s_\nu$ is indexed by a partition contained in $((r-k)^5)$.
	After specializing to $X=\{x_1,\ldots,x_5\}$, we have
	\[
	\bigl(s_\mu s_\nu\bigr)^{\vee,r}
	=
	s_\mu^{\vee,k}s_\nu^{\vee,r-k}.
	\]
\end{lemma}

\begin{proof}
	By the Littlewood--Richardson rule, after specializing to
	$X=\{x_1,\ldots,x_5\}$, we have
	\[
	s_\mu(X)s_\nu(X)
	=
	\sum_{\substack{\lambda\\ \ell(\lambda)\le5}}
	c_{\mu,\nu}^{\lambda}s_\lambda(X),
	\]
	since $s_\lambda(X)=0$ whenever $\ell(\lambda)>5$.
	Moreover, $\lambda_1\le \mu_1+\nu_1\le r$, and hence
	$\lambda\subseteq(r^5)$.
	Applying~\eqref{comple} to $s_\la(X),$ 
	\begin{align*}
		(s_\mu s_\nu)(X)&=\sum_{\substack{\la}}c_{\mu,\nu}^{\la}s_\la(X)=\sum_{\substack{\la}}c_{\mu,\nu}^{\la}\pi_5^r s^{\vee,r}_{\la}(X^{-1})\\
		&=\pi_5^r\left(\sum_{\substack{\la}}c_{\mu,\nu}^{\la} s^{\vee,r}_{\la}\right)(X^{-1})=\pi_5^r\left(\sum_{\substack{\la}}c_{\mu,\nu}^{\la} s_{\la}\right)^{\vee,r}(X^{-1})=\pi_5^r \left(s_{\mu} s_\nu\right)^{\vee,r}(X^{-1}).
	\end{align*}
	On the other hand, using~\eqref{comple} on $s_\mu(X)$ and $s_\nu(X)$, we have 
	$$
	(s_\mu s_\nu)(X)=s_\mu(X) s_\nu(X)=\pi_5^k s_\mu^{\vee,k}(X^{-1})\cdot \pi_5^{r-k}s_\nu^{\vee,r-k}(X^{-1})=\pi_5^{r}\left(s_\mu^{\vee,k}s_\nu^{\vee,r-k}\right)(X^{-1}).
	$$
	We complete the proof by comparing the two expressions.
\end{proof}

Therefore, by Lemma~\ref{slsm}, we have
\begin{equation*}
	F_i(r)\big|_{\ell\leq5}(X)=\pi_5^r\,G_i(r)(X^{-1}),
\end{equation*}
where $G_i(r)=\left(F_i(r)\big|_{\ell\leq5}\right)^{\vee,r}$ and
\begin{align*}
	G_1(r)&=\sum_{k=0}^{r}\Bigl[
	\bigl(s_{(k,k,2)}+s_{(k,k,1,1)}\bigr)\, s_{(r-k,r-k)}
	-s_{(k,k,1)}\, s_{(r-k,r-k,1)}\Bigr],
    \\
	G_2(r)&=\sum_{k=0}^{r}\Bigl[
	\bigl(s_{(k,k,2,1)}+s_{(k,k,1,1,1)}\bigr)\, s_{(r-k,r-k)}
	-s_{(k,k,1)}\,s_{(r-k,r-k,1,1)}\Bigr],\\
	G_3(r)&=\sum_{k=0}^{r}\Bigl[
	\bigl(s_{(k,k,2,2)}+s_{(k,k,2,1,1)}\bigr)\,s_{(r-k,r-k)}
	-s_{(k,k,1,1)}\,s_{(r-k,r-k,1,1)}\Bigr].
\end{align*}

\begin{theorem}
	The symmetric functions $F_i(r)\big|_{\ell\leq5}$ are Schur positive for $i=1,2,3$.
	In particular,
	\begin{align}
		F_1(r)\big|_{\ell\leq5}
		&=\sum_{j=0}^{\lfloor(r-1)/2\rfloor}
		s_{(r,r-j-1,r-j-1,j,j)},
		\label{G_1}\\
		F_2(r)\big|_{\ell\leq5}
		&=\sum_{j=0}^{\lfloor(r-1)/2\rfloor}
		s_{(r-1,r-j-1,r-j-1,j,j)},
		\label{G_2}\\
		F_3(r)\big|_{\ell\leq5}
		&=\sum_{j=1}^{\lfloor(r-1)/2\rfloor}
		s_{(r-2,r-j-1,r-j-1,j,j)}.
		\label{G_3}
	\end{align}
\end{theorem}

\begin{proof}
	We now give the Schur expansion of $G_i(r)$ for $i=1,2,3$:
	\begin{align*}
		G_1(r)&=
		\sum_{j=0}^{\lfloor(r-1)/2\rfloor}
		s_{(r-j,r-j,j+1,j+1)}.\\
		G_2(r)&=\sum_{j=0}^{\lfloor(r-1)/2\rfloor}
		s_{(r-j,r-j,j+1,j+1,1)}.\\
		G_3(r)&=\sum_{j=1}^{\lfloor(r-1)/2\rfloor}
		s_{(r-j,r-j,j+1,j+1,2)}.
	\end{align*}
	Since $\left(F_i(r)\big|_{\ell\leq5}\right)^{\vee,r}=G_i(r)$,
	the Schur expansions of $F_i(r)\big|_{\ell\leq5}$ follow by~\eqref{eq-com}.
    
    \noindent
    \medskip
	(i) \textbf{Schur expansion of $G_1(r)$.}
	Let 
	$$
	\mathcal{G}_1(t)=\sum_{r\ge 0}G_1(r) t^r
	=\sum_{r\ge 0}\sum_{k=0}^{r}\Bigl[
	\bigl(s_{(k,k,2)}+s_{(k,k,1,1)}\bigr)\,s_{(r-k,r-k)}
	-s_{(k,k,1)}\,s_{(r-k,r-k,1)}\Bigr]t^r.
	$$
	Then 
	$$
	\mathcal{G}_1(t)=\left(\Zb+\Zc\right)\, \Ya-\Za^2.
	$$
	
	By Theorem~\ref{h-e}, Pieri's rule gives $h_1\Ya=\Yb+\Za,$
	\begin{align*}
		h_1 \Za=\Zb+\Zc+\Zd, \quad
		h_1 \Yb=\Yc+\frac{\Ya-1}{t}+\Zd.
	\end{align*}
	A direct computation yields
	\begin{align*}
		\mathcal{G}_1(t)&=\left(h_1 \Za-\Zd\right)\, \Ya-\Za^2\\
		&=h_1 \Ya \Za-\Zd \Ya-\Za^2\\
		&=\left(\Yb+\Za\right) \Za-\Zd \Ya-\Za^2\\
		&=\Yb\Za-\Zd \Ya\\
		&=\Yb\left(h_1\Ya-\Yb\right)-\left(h_1 \Yb-\Yc-\frac{\Ya-1}{t}\right)\, \Ya\\
		&=\Ya\Yc-\Yb^2+\frac{\Ya^2-\Ya}{t}.
	\end{align*}
	Consequently‌,
	$$
	t \mathcal{G}_1(t)=t\cdot\left(\Ya\Yc-\Yb^2\right)+\Ya^2-\Ya.
	$$
	By Theorem~\ref{thm:base-pf}, we obtain
	$$
	t\mathcal{G}_1(t)=\sum_{0\le b\le a}s_{(a,a,b,b)}\,t^{a+b}-\Ya=\sum_{0\le b\le a}s_{(a,a,b,b)}\,t^{a+b}-\sum_{k\geq0}s_{(k,k)}\,t^{k}=t\sum_{1\le b\le a}s_{(a,a,b,b)}\,t^{a+b-1}.
	$$
	Taking the coefficient of $t^{r+1}$, and
	setting $b=j+1$ and $a=r-j$.
	Then 
	$$
	G_1(r)=\sum_{1\le j+1\le r-j}s_{(r-j,r-j,j+1,j+1)}=\sum_{j=0}^{\lfloor(r-1)/2\rfloor}
	s_{(r-j,r-j,j+1,j+1)}.
	$$

    \noindent
    \medskip
	(ii) \textbf{Schur expansion of $G_2(r)$.}
	Let
	$$
	\mathcal{G}_2(t)=\sum_{r\ge 0}G_2(r) t^r=\sum_{r\ge 0}\sum_{k=0}^{r}\Bigl[
	\bigl(s_{(k,k,2,1)}+s_{(k,k,1,1,1)}\bigr)\,s_{(r-k,r-k)}
	-s_{(k,k,1)}\,s_{(r-k,r-k,1,1)}\Bigr]t^r.
	$$
	Note that $e_1s_{(k,k,1,1)}=s_{(k+1,k,1,1)}+s_{(k,k,2,1)}+s_{(k,k,1,1,1)}$.
	Then 
	$$
	\mathcal{G}_2(t)=\left(e_1\Zc-\Ze\right) \Ya-\Za \Zc.
	$$
	Theorem~\ref{h-e} gives $e_1\Ya=\Yb+\Za$.
	Then
   $$
   \mathcal{G}_2(t)=\Zc\left(e_1\Ya-\Za\right)-\Ze\Ya
		=\Zc\Yb-\Ze\Ya.
   $$
    By Lemma~\ref{proof_j2}, 
	\begin{align*}
		e_2 \Ya-e_1 \Yb&=\Zc-\Yc,\\
		e_2 \Yb-e_1\Yc&=\Ze+\frac{\Za}{t}-\Zf.
	\end{align*}
    We defer the proof of Lemma~\ref{proof_j2} to Appendix B.
   Substituting $S_{(k,k,1,1)}$ and $S_{(k+1,k,1,1)}$ yields
	\begin{align*}
		&\mathcal{G}_2(t)=\left(e_2 \Ya-e_1 \Yb+\Yc\right) \Yb\\
        &\qquad \qquad-\left(e_2 \Yb-e_1\Yc-\frac{\Za}{t}+\Zf\right) \Ya\\
		&\qquad=-e_1\Yb^2+\Yb\Yc+e_1\Ya\Yc+\frac{\Ya\Za}{t}-\Ya\Zf\\
		&\qquad=-e_1\Yb^2+\Yb\Yc+e_1\Ya\Yc\\
        &\qquad\qquad+\frac{\Ya\left(e_1\Ya-\Yb\right)}{t}-\Ya\Zf.
	\end{align*}
	Thus, $t\mathcal{G}_2(t)$ equals
	\begin{align*}
		e_1 \left[\Ya^2+t\cdot \left(\Ya\Yc-\Yb^2\right)\right]+t\cdot \left(\Yb\Yc- \Ya\Zf\right)-\Ya\Yb.
	\end{align*}
	Combining~\eqref{eq:base-pf} in Theorem~\ref{thm:base-pf} with~\eqref{eq:base-pf2} in Theorem~\ref{thm:base-pf22}, we obtain 
	\begin{equation*}
		t\mathcal{G}_2(t)=
		\sum_{0\le b\le a} s_{(a+1,a+1,b+1,b+1,1)}\,
		t^{a+b+2}.	
	\end{equation*}
	Taking the coefficient of $t^{r+1}$, and
	setting $b=j$, $a=r-j-1$. Then
	$$
	G_2(r)=\sum_{0\le j\le r-j-1}s_{(r-j,r-j,j+1,j+1,1)}=\sum_{j=0}^{\lfloor(r-1)/2\rfloor}
	s_{(r-j,r-j,j+1,j+1,1)}.
	$$

     \noindent
    \medskip
	(iii) \textbf{Schur expansion of $G_3(r)$.}
	Let
	$$
	\mathcal{G}_3(t)=\sum_{r\ge 0}G_3(r) t^r=\sum_{r\ge 0}\sum_{k=0}^{r}\Bigl[
	\bigl(s_{(k,k,2,2)}+s_{(k,k,2,1,1)}\bigr)s_{(r-k,r-k)}
	-s_{(k,k,1,1)}s_{(r-k,r-k,1,1)}\Bigr]\,t^r.
	$$
    Then we have
	$$
	\mathcal{G}_3(t)=\left(\Zg+\Zh\right) \Ya-\Zc^2.
	$$
	Since all identities in this part are considered after
specialization to the five-variable alphabet
$X=\{x_1,\ldots,x_5\}$, we have $s_\lambda(X)=0$
whenever $\ell(\lambda)>5$,
			and
			$$
			e_5s_\lambda=s_{\lambda+(1^5)}.
			$$
	This gives
			$$
			e_5t\cdot \Za
			=\sum_{k\geq0}s_{(k+1,k+1,2,1,1)}\,t^{k+1}=\Zh,
			$$
    and
   $$
			e_5t\cdot \Ya
			=\sum_{k\geq0}s_{(k+1,k+1,1,1,1)}\,t^{k+1}=S_{(k,k,1,1,1)}.
			$$
    Hence,
 $$
 \Zh\Ya=e_5t\cdot S_{(k,k,1)} S_{(k,k)} =S_{(k,k,1)} S_{(k,k,1,1,1)}.
 $$
Thus 
   $$
	\mathcal{G}_3(t)=\Zg \Ya+S_{(k,k,1)} S_{(k,k,1,1,1)}-\Zc^2.
	$$
	By~\eqref{jr3} in Theorem~\ref{thm:base-pf22}, we have
	\begin{equation*}
		\mathcal{G}_3(t)=\sum_{a\geq b\geq0}s_{(a+2,a+2,b+2,b+2,2)}\,t^{a+b+3}.
	\end{equation*}
	Taking the coefficient of $t^{r}$, and
	setting $b=j-1$, $a=r-j-2$. Then
	$$G_3(r)=\sum_{0\le j-1\le r-j-2}s_{(r-j,r-j,j+1,j+1,2)}=\sum_{j=1}^{\lfloor(r-1)/2\rfloor}
	s_{(r-j,r-j,j+1,j+1,2)}.$$
	This completes the proof.
\end{proof}

\subsubsection{The length-six terms}
As stated in Section~\ref{subsec:<=5}, for every term $s_{\mu} s_{\nu}$ in $\omega\left(\sum_{k=0}^{r}F(r,k)\right)$,  
the length of each partition in its Littlewood--Richardson expansion is at most $6$.
In this part, we study the relationship between the Schur expansion of $F_i(r)$
and that of $F_i(r)\big|_{\ell\leq5}$ for $i=1,2,3$.

For a partition
$\la=(\la_1,\ldots,\la_{\ell(\la)})$, let
$$
\la^{-}
=(\la_1-1,\ldots,\la_{\ell(\la)}-1)
$$
denote the partition obtained by deleting the first column of its
Young diagram, with zero parts omitted. 
Recall that $c_{\mu,\nu}^{\lambda}$ is the Littlewood--Richardson
coefficient of $s_\mu s_\nu$.
We introduce a result concerning $c_{\mu,\nu}^{\lambda}$ and $c_{\mu^-,\nu^-}^{\lambda^-}$.

\begin{lemma}{\rm \cite[Lemma A.1]{CGH+24}}
\label{lem:first-column}
	If $c_{\mu,\nu}^{\lambda}>0$ and
	$\ell(\lambda)=\ell(\mu)+\ell(\nu)$, then 
	$$	
	c_{\mu,\nu}^{\lambda}
	=
	c_{\mu^{-},\nu^{-}}^{\lambda^{-}}.
	$$
\end{lemma}
Thus, in the maximal-length case, simultaneously deleting the first
column from $\mu$, $\nu$, and $\lambda$ preserves the corresponding
Littlewood--Richardson coefficient.
With a simple augmentation of
Littlewood--Richardson tableaux, we obtain the following
adding-column result.

\begin{lemma}\label{lem:column-shift}
	Suppose that  $\ell(\mu),\ell(\nu)\le 3$ and $\ell(\lambda)\le 6$. Then for every
	$d\geq0$,
	\begin{equation*}
		c_{\mu+(d^3),\nu+(d^3)}^{\lambda+(d^6)}
		=c_{\mu,\nu}^{\lambda}.
	\end{equation*}
\end{lemma}

\begin{proof}
	It suffices to consider $d=1$. 
	Suppose that 
	$c_{\mu+(1^3),\nu+(1^3)}^{\lambda+(1^6)}>0$.
	Then, by Lemma~\ref{lem:first-column},
	$$c_{\mu,\nu}^{\lambda}=c_{\mu+(1^3),\nu+(1^3)}^{\lambda+(1^6)}.$$
	Now we claim that if $c_{\mu+(1^3),\nu+(1^3)}^{\lambda+(1^6)}=0$,
	then $c_{\mu,\nu}^{\lambda}=0$.
	
	Assume, to the contrary, that $c_{\mu,\nu}^{\lambda}>0$, and choose a
	Littlewood--Richardson tableau $T$ of shape
	$\lambda/\mu$ and content $\nu$.
	Add one column to the left of $T$. 
	The skew diagram
	$$
	({\lambda+(1^6)})/({\mu+(1^3)})
	$$
	then has three additional cells, namely the cells in the
	first column of rows $4,5,6$. Fill these cells with
	$1,2,3$, respectively, and denote the resulting tableau
	by $T^+$.
	
	The new first-column entries are strictly increasing.
	Moreover, since $\ell(\mu)\leq 3$, every old entry in
	row $5$ is at least $2$, and every old entry in row
	$6$ is at least $3$. Thus the rows of $T^+$ remain
	weakly increasing.
	In the reverse reading word, the newly inserted entries occur
	at the ends of rows $4,5,6$, respectively. Hence the
	successive additional contributions to the multiplicities
	of $1,2,3$ are
	$
	(1,0,0),\,(1,1,0),\,(1,1,1),
	$
	so the lattice-word inequalities are preserved. Therefore
	$T^+$ is a Littlewood--Richardson tableau of shape
	$({\lambda+(1^6)})/({\mu+(1^3)})$
	and content $\nu+(1^3)$. Consequently,
	$$
	c_{\mu+(1^3),\,\nu+(1^3)}^{\lambda+(1^6)}>0,
	$$
	which is a contradiction. This completes the proof.
\end{proof}

Let
\begin{equation*}
	\lambda=(\lambda_1,\lambda_2,\lambda_3,\lambda_4,\lambda_5,\lambda_6),
\end{equation*}
where $\lambda_6>0$.
For each term $s_\mu s_\nu$ in $F_1(r)$, 
\begin{itemize}
	\item[(i)] If $\ell(\mu)+\ell(\nu) \le 5$, then $[s_\lambda]s_{\mu}s_{\nu}=0$.
	\item[(ii)] If $\ell(\mu)+\ell(\nu)=6$, then, by Lemma~\ref{lem:column-shift}, we have $c^{\lambda}_{\mu,\nu}=c^{\lambda-(1^6)}_{\mu-(1^3),\nu-(1^3)}$, that is, $$[s_\lambda]s_{\mu}s_{\nu}=[s_{\lambda-(1^6)}]s_{\mu-(1^3)}\,s_{\nu-(1^3)}.$$
\end{itemize}
Recall that
\begin{align*}
	F_1(r)=\sum_{k=0}^{r}
	\left[\left(s_{(k,k,k-2)}+s_{(k,k-1,k-1)}\right)\, s_{(r-k,r-k,r-k)}-s_{(k,k,k-1)}\, s_{(r-k,r-k,r-k-1)}\right],
\end{align*}
As an example, consider $\sum_{k=0}^{r} s_{(k,k,k-2)}\, s_{(r-k,r-k,r-k)}$. Then
$$
\sum_{k=0}^{r} s_{(k,k,k-2)}\, s_{(r-k,r-k,r-k)}=\sum_{k=0}^{r} \sum_{\lambda} c_{(k,k,k-2),(r-k,r-k,r-k)}^{\lambda}\, s_{\la}.
$$
Note that 
\begin{align*}
	\sum_{k=0}^{r-2} s_{(k,k,k-2)}\, s_{(r-2-k,r-2-k,r-2-k)}&=\sum_{k=1}^{r-1} s_{(k-1,k-1,k-3)}\, s_{(r-k-1,r-k-1,r-k-1)}\\
	&=\sum_{k=0}^{r} \sum_{\lambda} c_{(k-1,k-1,k-3),(r-k-1,r-k-1,r-k-1)}^{\lambda}\, s_{\la}.
\end{align*}
By Lemma~\ref{lem:column-shift},
$$
c_{(k,k,k-2),(r-k,r-k,r-k)}^{\lambda}=c_{(k-1,k-1,k-3),(r-k-1,r-k-1,r-k-1)}^{\la-(1^6)}.
$$
Then
$$
\left[s_\lambda\right]\sum_{k=0}^{r} s_{(k,k,k-2)}\, s_{(r-k,r-k,r-k)}=\left[s_{\lambda-(1^6)}\right]\sum_{k=0}^{r-2} s_{(k,k,k-2)}\, s_{(r-2-k,r-2-k,r-2-k)}.
$$
The same holds for 
$$s_{(k,k-1,k-1)}\, s_{(r-k,r-k,r-k)} \quad \textrm{and}\quad s_{(k,k,k-1)}\, s_{(r-k,r-k,r-k-1)}.$$
Hence, we obtain
\begin{align}\label{deletcolum}
	\left[s_\lambda\right] F_1(r)=\left[s_{\lambda-(1^6)}\right] F_1(r-2).
\end{align}
Similarly, we have $[s_\lambda] F_i(r)=[s_{\lambda-(1^6)}] F_i(r-2)$ for $i=2,3$.

Let $d=\lambda_6$. Applying \eqref{deletcolum} repeatedly gives
\begin{equation*}
	\left[s_\lambda\right] F_1(r)
	=
	\left[s_{\lambda-(1^6)}\right] F_1(r-2)
	=
	\cdots
	=
	\left[s_{\tilde{\lambda}}\right] F_1(r-2d),
\end{equation*}
where $\tilde{\lambda}=(\lambda_1-d,\lambda_2-d,\lambda_3-d,\lambda_4-d,\lambda_5-d,0)$.
By~\eqref{G_1}, the partitions occurring in $F_1(r-2d)\big|_{\ell\leq5}$ are precisely
$
    (r-2d,r-2d-j-1,r-2d-j-1,j,j)
$
for $0\le j\le \lfloor(r-2d-1)/2\rfloor$, and each occurs with coefficient 1.

On the other hand, from \eqref{deletcolum}, we also obtain 
\begin{equation*}
	\left[s_{\tilde{\lambda}+(1^6)}\right] F_1(r-2d+2)
	=
	\left[s_{\tilde{\lambda}}\right] F_1(r-2d).
\end{equation*}
Therefore, by adding back these $d$ columns of length $6$, 
i.e., adding $d$ to each part, we have 
\begin{equation*}
	F_1(r)
	=\sum_{j=0}^{\lfloor(r-2d-1)/2\rfloor} \sum_{d}
	s_{(r-d,r-d-j-1,r-d-{j}-1,{j}+d,{j}+d,d)}.
\end{equation*}
Setting \(j'=j+d\) and then
renaming \(j'\) as \(j\),  we obtain
\begin{equation}\label{GG1}
	F_1(r)=\sum_{0\le d\le j\le \lfloor(r-1)/2\rfloor}
	s_{(r-d,r-j-1,r-j-1,j,j,d)}.
\end{equation}

Applying the same argument to~\eqref{G_2} and~\eqref{G_3}, respectively,
gives
\begin{equation}\label{GG2}
	F_2(r)=
	\sum_{0\leq d\leq j\leq\lfloor(r-1)/2\rfloor}
	s_{(r-1-d,r-j-1,r-j-1,j,j,d)},
\end{equation}
and
\begin{equation}\label{GG3}
	F_3(r)=
	\sum_{0\leq d < j\leq\lfloor(r-1)/2\rfloor}
	s_{(r-2-d,r-j-1,r-j-1,j,j,d)}.
\end{equation}

\begin{proof}[Proof of Theorem~\ref{thm:main}]
	By~\eqref{GG1},~\eqref{GG2} and~\eqref{GG3}, we obtain that $F_1(r), F_2(r)$ and $F_3(r)$ are Schur positive.
Hence, by~\eqref{omega1}, 
$$
\omega\left(\sum_{k=0}^{r}F(r,k)\right)=F_1(r)+2\,F_2(r)+F_3(r)
$$
	is Schur positive.
	The Schur positivity of $\sum_{k=0}^{r}F(r,k)$ follows by Lemma~\ref{SPOMEGA}.
\end{proof}
\section{Log-convexity preservation and $q$-log-concavity}\label{sec:baxter-transformation}

In this section, we establish two further logarithmic properties
associated with the refined Baxter numbers. We first prove that the
Baxter transformation preserves log-convexity. We then study the
$q$-log-concavity of the $q$-refined Baxter numbers in both the row
and column directions, and extend these results to the
$q$-analogues of the $d$-Hoggatt numbers.

\subsection{The Baxter transformation}
We show that the Baxter transformation preserves log-convexity.
The $q$-log-convexity of a polynomial sequence is closely related to 
the preservation of log-convexity by linear transformations.
For a double-indexed sequence $A(n,k)$ with $0\le k\le n$
and a sequence of nonnegative numbers $(a_k)_{k\ge 0},$
define 
\begin{equation*}
b_n=\sum_{k=0}^{n}A(n,k)a_k.
\end{equation*}
Note that in this section, the double-indexed sequence is indexed from $0$ in both parameters.
Liu and Wang~\cite{LW07} systematically studied linear transformations that preserve log-convexity.
They proved that the linear transformations associated with many combinatorial numbers, such as the binomial coefficients, the Stirling numbers of the second kind and the Eulerian numbers preserve log-convexity, i.e., the log-convexity of $(a_k)_{k\ge 0}$ implies that of $(b_n)_{n\ge 0}.$
See~\cite{CWY11,LW07,WZ16} for more details.
Chen, Wang and Yang~\cite{CWY10} subsequently proved that the linear transformation defined by the Narayana numbers preserves log-convexity.
We present the following result.

\begin{theorem}\label{thm:baxter-transform}
If the sequence $(a_k)_{k\geq 0}$ of nonnegative real numbers is
log-convex, then so is the sequence $(b_n)_{n\geq 0}$, where
$$
b_n=\sum_{k=0}^{n}B(n,k)a_k,\quad n\geq 0.
$$
\end{theorem}
Following Liu and Wang~\cite[Section 4.2]{LW07},
it is not difficult to verify that 
for $n\geq 1$,
\begin{equation}\label{b_n}
	b_{n-1}b_{n+1}-b_n^2
	=
	\sum_{r=0}^{2n}
	\left(
	\sum_{k=0}^{\left\lfloor r/2\right\rfloor}
	\beta_k(n,r)a_ka_{r-k}
	\right),
\end{equation}
and
\begin{equation}\label{lcxcoeff}
	B_{n-1}(x)B_{n+1}(x)-B_n(x)^2
	=
	\sum_{r=0}^{2n}
	\left(
	\sum_{k=0}^{\left\lfloor r/2\right\rfloor}
	\beta_k(n,r)
	\right)x^r,
\end{equation}
where 
\begin{equation*}
	\beta_k(n,r)= 
	\begin{cases}
		B(n-1,k)B(n+1,r-k)+B(n-1,r-k)B(n+1,k)\\
		\qquad-2B(n,k)B(n,r-k) & \text{if } 0\leq k<r/2,\\
		B(n-1,k)B(n+1,k)-B^2(n,k) & \text{if } k=r/2.
	\end{cases}
\end{equation*}
Recall that 
$$
B(n,k)=\frac{\binom{n}{k}\binom{n+1}{k}\binom{n+2}{k}}
	{\binom{k+1}{k}\binom{k+2}{k}}.
$$
The following result can be obtained by direct computation; see~\cite[Theorem 2.3]{MS26}.
\begin{lemma}\label{lc-n}
	Given an integer $k$, the sequence $(B(n,k))_{n\ge k}$  is log-concave.
\end{lemma}

\begin{proof}[Proof of Theorem~\ref{thm:baxter-transform}]
	Since the sequence $(a_k)_{k\geq0}$ is nonnegative and log-convex, we have
	$$
	a_0a_r\geq a_1a_{r-1}\geq a_2a_{r-2}\geq\cdots.
	$$
We claim that, for every fixed $n$ and $r$, there exists an integer
$k'=k'(n,r)$ such that
\[
\beta_k(n,r)\geq 0
\quad\text{for }0\leq k\leq k'
\qquad\textrm{and}\qquad
\beta_k(n,r)\leq 0
\quad\text{for }k'<k\leq \left\lfloor\frac r2\right\rfloor.
\]
	By Theorem~\ref{q-lcx}, $(B_n(x))_{n\ge 0}$ is $q$-log-convex.
    Then by~\eqref{lcxcoeff},  
	$$
	\sum_{k=0}^{\left\lfloor r/2\right\rfloor}\beta_k(n,r)\ge 0
	$$
	for all $r$. 
	Therefore,
	\begin{align*}
		\sum_{k=0}^{\left\lfloor r/2\right\rfloor}
		\beta_k(n,r)a_ka_{r-k}=&\sum_{k=0}^{k'}
		\beta_k(n,r)a_ka_{r-k}+\sum_{k=k'+1}^{\left\lfloor r/2\right\rfloor}
		\beta_k(n,r)a_ka_{r-k}\\
		\ge&\sum_{k=0}^{k'}
		\beta_k(n,r)a_{k'}a_{r-k'}+\sum_{k=k'+1}^{\left\lfloor r/2\right\rfloor}
		\beta_k(n,r)a_{k'}a_{r-k'}\\
		={}&a_{k'}a_{r-k'} \sum_{k=0}^{\left\lfloor r/2\right\rfloor}
		\beta_k(n,r)\ge 0.
	\end{align*}
	By~\eqref{b_n}, for each $n$, $b_{n-1}b_{n+1}-b_n^2\ge 0$ and the sequence $(b_n)_{n\geq0}$ is log-convex. 
	
	We now prove the existence of \(k'\).    
	Note that for
	$r=0$, $$\beta_0({n,0})=B(n-1,0)B(n+1,0)-B(n,0)^2=0.$$
    We assume that $r\ge 1$.
    Let 
	\begin{equation}
		\beta'_k(n,r)= 
		\begin{cases}
			2\beta_k(n,r) & \text{if $r$ is even and $k=r/2$},\\
			\beta_k(n,r) & \text{otherwise}.
		\end{cases}
	\end{equation}
   Clearly, $\beta'_k(n,r)$ and $\beta_k(n,r)$ have the same sign. 
   For $0\le k\le \lfloor{ r/2\rfloor}$,
   $$
   \beta'_k(n,r)=B(n-1,k)B(n+1,r-k)+B(n-1,r-k)B(n+1,k)-2\,B(n,k)B(n,r-k).
   $$
    We divide the range of $k$ into three cases: 
   \begin{equation}\label{interval}
    k\le r-n-1,\qquad r-n\le k<r/2,\qquad k=r/2 \quad (\textrm{if $r$ is even}).
    \end{equation}

    For the first interval $k\le r-n-1$,
     \begin{equation}\label{rgen}
    \beta'_k(n,r)=B(n-1,k)B(n+1,r-k)
	+
	B(n+1,k)B(n-1,r-k)\ge 0 
    \end{equation}
    for $k\le r-n-1$,
    since $B(n,r-k)=0$. 
    
    For the second interval,
	consider $\max\{0,r-n\}\le k<r/2$. 
	Then $\beta'_k(n,r)$ equals
	\begin{align*}
		&B(n,k)B(n,r-k) \left( 
		\frac{B(n-1,k)B(n+1,r-k)}{B(n,k)B(n,r-k)}+
		\frac{B(n+1,k)B(n-1,r-k)}{B(n,k)B(n,r-k)}-
		2 \right) \\
		={}&B(n,k)B(n,r-k) \left(\frac{\binom{n-1}{k}}{\binom{n+2}{k}} \frac{\binom{n+3}{r-k}}{\binom{n}{r-k}}
		+\frac{\binom{n+3}{k}}{\binom{n}{k}} \frac{\binom{n-1}{r-k}}{\binom{n+2}{r-k}}
		-2\right)\\
		={}&B(n,k)B(n,r-k) \left( \frac{n+3}{n} \left(
		\prod_{i=0}^{2}
		\frac{n-k+i}{n-r+k+i+1}+
		\prod_{j=0}^{2}
		\frac{n-r+k+j}{n-k+j+1} \right)
		-2 \right).
	\end{align*}
	Let 
	\begin{equation*}
		T(k)=\prod_{i=0}^{2}
		\frac{n-k+i}{n-r+k+i+1}+
		\prod_{j=0}^{2}
		\frac{n-r+k+j}{n-k+j+1}.
	\end{equation*}
	Then 
	$$
	\beta'_k(n,r)=B(n,k)B(n,r-k)\left( \frac{n+3}{n}T(k)-2 \right).
	$$ 
    It is clear that $\beta'_k(n,r)$ has the same sign as  $\frac{n+3}{n}T(k)-2$ for each $k$.
    
    Now we prove $T(k)-T(k+1)\ge 0$.
    Note that 
    \begin{equation*}
		T(k+1)=\prod_{i=0}^{2}
		\frac{n-k-1+i}{n-r+k+i+2}+
		\prod_{j=0}^{2}
		\frac{n-r+k+j+1}{n-k+j}.
	\end{equation*}
	For convenience, set
	$a=n-k$ and $b=n-r+k$.
	Then $a\ge 1,\, b\ge 0$.
Define
	\begin{align*}
	    T_1(k)&:=\frac{a(a+1)(a+2)}{(b+1)(b+2)(b+3)},\qquad T_2(k):=\frac{b(b+1)(b+2)}{(a+1)(a+2)(a+3)},\\
        T_1(k+1)&:=\frac{(a-1) a (a+1)}{(b+2)(b+3)(b+4)},\qquad T_2(k+1):=\frac{(b+1)(b+2)(b+3)}{a(a+1)(a+2)}.
	\end{align*}
    Then $T(k)=T_1(k)+T_2(k)$ and $T(k+1)=T_1(k+1)+T_2(k+1)$.

A direct calculation gives
\begin{align*}
T_1(k)-T_1(k+1)
&=
\frac{
3(a+b+3)a(a+1)
}{
(b+1)(b+2)(b+3)(b+4)
},\\
T_2(k+1)-T_2(k)
&=
\frac{
3(a+b+3)(b+1)(b+2)
}{
a(a+1)(a+2)(a+3)
}.
\end{align*}
Consequently,
\begin{equation}\label{eq:T12factor}
		\frac{T_1(k)-T_1({k+1})}{T_2({k+1})-T_2(k)}
		=
		\frac{a(a+1)}{(b+1)(b+2)}
		\cdot
		\frac{a(a+1)(a+2)(a+3)}
		{(b+1)(b+2)(b+3)(b+4)}.
\end{equation}
	Since $k<r/2$, 
    $$
    a-b=(n-k)-(n-r+k)=r-2k\ge1.
    $$
	So every factor in the right-hand side of \eqref{eq:T12factor} is
	greater than or equal to $1$. Hence
	$$
	T_1(k)-T_1({k+1})\ge T_2({k+1})-T_2(k).
	$$
	It follows that
	$$
	T(k)-T(k+1)
	=(T_1({k})-T_1(k+1))-(T_2({k+1})-T_2(k))\ge0.
	$$
	Thus, $T(k)$ is nonincreasing in $k$.

For the third interval, $k=r/2$.
Then 
$$\beta'_{r/2}(n,r)=2\left(B(n-1,r/2) B(n+1,r/2)-B(n,r/2)^2\right).$$
By Lemma~\ref{lc-n},
$\beta'_{r/2}(n,r)\le 0.$
Thus to prove that the index $k'$ exists,
it suffices to prove that $\beta'_0(n,r)\ge 0.$

\noindent
\medskip
\textbf{Case 1:} For $r\le n$, 
by~\eqref{interval}, the range of $k$ is divided into two intervals $0\le k< r/2$
and $r/2$.
Then 
    \begin{align*}
    \frac{n+3}{n}T(0)-2
    &=\prod_{i=0}^{2}\frac{n+1+i}{n-r+1+i}+\prod_{j=0}^{2}\frac{n-r+j}{n+j}-2\\
    &=\prod_{i=0}^{2}\left(1+\frac{r}{n-r+1+i}\right)+\prod_{j=0}^{2}\left(1-\frac{r}{n+j}\right)-2\\
    &\ge 1+\sum_{i=0}^{2}\frac{r}{n-r+1+i}+1-\sum_{j=0}^{2}\frac{r}{n+j}-2\\
    &=\sum_{i=0}^{2}\left(\frac{r}{n-r+1+i}-\frac{r}{n+i} \right)\ge 0,
    \end{align*}
    where the inequality follows from the two elementary inequalities $\prod_i (1+x_i)\ge 1+\sum_ix_i$ and
    $$
    \prod_i (1-x_i)\ge 1-\sum_i x_i 
    \qquad \textrm{for }\quad 0\le x_i\le 1.$$

Since $\frac{n+3}{n}T(0)-2$ and $\beta'_0(n,r)$ have the same sign,
and $T(k)$ is nonincreasing, the sign of $\beta'_k(n,r)$ can change at most once from nonnegative to nonpositive.

\noindent
\medskip
\textbf{Case 2:} For $n<r\le 2n$, 
by~\eqref{interval}, the range of $k$ is divided into three intervals $k\le r-n-1$, $r-n\le k< r/2$ and $k=r/2$.
By~\eqref{rgen},  $\beta'_k(n,r)\ge 0$ in the first interval $k\le r-n-1$.
Since  $T(k)$ is nonincreasing in the second interval $r-n\le k\le r/2$,
the sign of $\beta'_k(n,r)$ can change at most once.
Thus, in both cases, the index $k'$ exists. This completes the proof.
\end{proof}
\subsection{$q$-Log-concavity}\label{sec:q-log-concavity1}
In this subsection, we prove the $q$-log-concavity in
both the row and column directions for the
$q$-refined Baxter numbers. 
Chen, Wang and Yang~\cite{CWY10} established the $q$-log-concavity of the 
$q$-Narayana numbers.
We prove the following results on the $q$-refined Baxter numbers, in parallel to the results of $q$-Narayana numbers.
\begin{theorem}\label{q-lc-k}
	Given an integer $n$, the polynomial sequence $(B_q(n,k))_{0\le k\le n}$ is $q$-log-concave.
\end{theorem}

\begin{proof}
	By Lemma~\ref{BqSchur}, for any $k\ge 1$, 
	\begin{align*}
		B_q(n,k)^2-B_q(n,k-1)B_q(n,k+1)&=\ps_n\left(s_{(3^k)}\right)\ps_n\left(s_{(3^k)}\right)-\ps_n\left(s_{(3^{k+1})}\right)\ps_n\left(s_{(3^{k-1})}\right)\\
		&=\ps_n\left(s_{(3^k)} s_{(3^k)}-s_{(3^{k+1})}s_{(3^{k-1})}\right).
	\end{align*}

    Kirillov~\cite{Kir84} proved that for $a,b\ge 1$, we have the following identity for Schur functions
	\begin{equation}\label{kir}
	s_{(a^b)} s_{(a^b)}=
	s_{(a^{b+1})}s_{(a^{b-1})}+
	s_{((a+1)^{b})}s_{((a-1)^{b})}.
	\end{equation}
	Then 
    $$
    s_{(3^k)} s_{(3^k)}-s_{(3^{k+1})}s_{(3^{k-1})}=s_{(4^{k})}s_{(2^{k})}
    $$
    is Schur positive for any $k\ge0$.
	Hence, the sequence $\left(B_q(n,k)\right)_{0\le k\le n}$ is $q$-log-concave.
\end{proof}

The following is the $q$-analogue of Lemma~\ref{lc-n}.
\begin{theorem}\label{q-lc-n}
	Given an integer $k$, the polynomial sequence $(B_q(n,k))_{n\ge k}$  is $q$-log-concave.
\end{theorem}


We now turn to the $q$-log-concavity of the $q$-refined Baxter numbers for fixed $k$.
We first need the following result.

		
		\begin{lemma}\label{lem:min-degree}
			Let $r,m,k\geq 0$.  Then
			\begin{equation*}
				q^{-m\binom r2} \ps_{k+3}\left(s_{(m^r)}\right)\ge_q 0.
			\end{equation*}
		\end{lemma}
		\begin{proof}
        Note that if $r>k+3$, then
$
\operatorname{ps}_{k+3}(s_{(m^r)})=0,
$
and the assertion is immediate. We may therefore assume
$r\le k+3$.
			Using Lemma~\ref{lem:hook-content}, we have
			\begin{equation*}
			\ps_{k+3}\left(s_{(m^r)}\right)=q^{m\binom{r}{2}}\prod_{(i,j)\in(m^r)}
				\frac{1-q^{k+3+j-i}}{1-q^{h(i,j)}}.
			\end{equation*}
			It is known~\cite[p. 44]{Mac95} that $\prod_{(i,j)\in(m^r)}
			\frac{1-q^{k+3+j-i}}{1-q^{h(i,j)}}$ is a polynomial in $q$ with nonnegative coefficients.
			Hence, 
			\begin{equation*}
				q^{-m\binom r2}s_{(m^r)}(1,q,\ldots,q^{k+2})=\prod_{(i,j)\in(m^r)}
				\frac{1-q^{k+3+j-i}}{1-q^{h(i,j)}} \ge_q 0.
			\end{equation*}
			This completes the proof.
		\end{proof}
		
		\begin{proof}[Proof of Theorem~\ref{q-lc-n}]
			Recall that by~\eqref{eq-expre},
			\begin{align*}
				B_q(k+m,k)&=\ps_{k+m}\left(s_{(3^k)}\right)=q^{\frac{3k(k-1)}{2}}\prod_{i=1}^k \prod_{j=1}^3
				\frac{[k+m-i+j]_q}{[k-i-j+4]_q}.
			\end{align*}
			By reversing the index $i$ in the product formula,
			\begin{align*}
				B_q(k+m,k)&=q^{\frac{3k(k-1)}{2}}\prod_{i=1}^k 
				\frac{[k+m-i+1]_q\, [k+m-i+2]_q \,[k+m-i+3]_q}{[k-i+3]_q \,[k-i+2]_q \,[k-i+1]_q}\\
				&=q^{\frac{3k(k-1)}{2}}\prod_{i=1}^{k}
				\frac{[m+i]_q\,[m+i+1]_q\,[m+i+2]_q}
				{[i]_q\,[i+1]_q\,[i+2]_q}\\
				&=q^{\frac{3k(k-1)}{2}} \cdot \frac{[k+m]_q!}{[k]_q! [m]_q!} \cdot
				\frac{[k+m+1]_q! [1]_q!}{[k+1]_q! [m+1]_q!}\cdot
				\frac{[k+m+2]_q! [2]_q!}{[k+2]_q! [m+2]_q!}\\
				&=q^{\frac{3k(k-1)}{2}}\prod_{j=1}^{3}
				\frac{[k+m+j-1]_q!\,[j-1]_q!}
				{[k+j-1]_q!\,[m+j-1]_q!}.
			\end{align*}

			On the other hand, 
			by Lemma~\ref{lem:hook-content},
			\[
			\ps_{k+3}\left(s_{(m^3)}\right)
			=
			q^{3m}\prod_{i=1}^{3}\prod_{j=1}^{m}
			\frac{[k+3-i+j]_q}
			{[m-i-j+4]_q}=q^{3m} \prod_{i=1}^{3}
			\frac{[k+m+3-i]_q!\,[3-i]_q!}
			{[k+3-i]_q!\,[m+3-i]_q!}.
			\]
			Substituting $j=4-i$ yields
			$$\ps_{k+3}\left(s_{(m^3)}\right)
			=
			q^{3m}
			\prod_{j=1}^{3}
			\frac{[k+m+j-1]_q!\,[j-1]_q!}
			{[k+j-1]_q!\,[m+j-1]_q!}.
			$$
			Consequently,
			\[
			B_q(k+m,k)
			=
			q^{\frac{3k(k-1)}{2}-3m}\ps_{k+3}\left(s_{(m^3)}\right).
			\]
			
			For $m=0$, we have  
			\begin{equation*}
				B_q(k,k)^2-B_q(k-1,k)B_q(k+1,k)=B_q(k,k)^2=q^{3k(k-1)}\ge_q 0.
			\end{equation*}
			For $m\geq 1$,
			\begin{align*}
				&B_q(k+m,k)^2
				-
				B_q(k+m-1,k)B_q(k+m+1,k)\\
				={}&
				q^{3k(k-1)-6m}\cdot
				\ps_{k+3}\left(s_{(m^3)}^2-s_{((m-1)^3)}s_{((m+1)^3)}
				\right).
			\end{align*}
	        Taking $a=m, b=3$ in~\eqref{kir} yields
			$$
			B_q(k+m,k)^2
			-
			B_q(k+m-1,k)B_q(k+m+1,k)
			=
			q^{3k(k-1)-6m}\cdot
			\ps_{k+3}\left(s_{(m^2)}s_{(m^4)}
			\right).
			$$
			Then, by Lemma~\ref{lem:min-degree}, 
			\begin{align*}
				&B_q(k+m,k)^2-B_q(k+m-1,k)B_q(k+m+1,k)\\
				={}&q^{3k(k-1)+m}
				\left( q^{-m}\ps_{k+3}\left(s_{(m^2)}
				\right) \right)
				\left( q^{-6m}\ps_{k+3}\left(s_{(m^4)}
				\right) \right)\ge_q 0, 
			\end{align*}
			which implies the $q$-log-concavity of $\left(B_q(n,k)\right)_{n\ge k}$.
		\end{proof}	

\subsection{$q$-Analogue of the $d$-Hoggatt numbers}
		Define the $q$-analogue of the $d$-Hoggatt numbers
		\begin{equation*}
			H^{(d)}_{q}(n,k)=q^{\frac{dk(k-1)}{2}} \prod_{j=0}^{d-1} \frac{\begin{bmatrix} n+j \\ k \end{bmatrix}_q}{\begin{bmatrix} k+j \\ k \end{bmatrix}_q}.
		\end{equation*}
		In particular, when $d=1, 2$, and $3$, $H^{(d)}_{q}(n,k)$ specializes to the shifted $q$-binomial coefficients, the $q$-Narayana numbers, and the $q$-refined Baxter numbers, respectively.
		By an argument analogous to the proof of  Lemma~\ref{BqSchur}, we can verify that
		$$
		H^{(d)}_{q}(n,k)=\ps_n\left(s_{(d^k)}\right).
		$$
		The following results generalize Theorems~\ref{q-lc-k} and~\ref{q-lc-n}.
		\begin{theorem}
			Given an integer $n$, the polynomial sequence $\left(H^{(d)}_{q}(n,k)\right)_{0\le k\le n}$ is $q$-log-concave for any $d\ge1$.
		\end{theorem}
       \begin{proof}
           The proof is similar to that of Theorem~\ref{q-lc-k}, with $a=d$ in~\eqref{kir}.
       \end{proof}
        
		\begin{theorem}
			Given an integer $k$, the polynomial sequence $\left(H^{(d)}_{q}(n,k)\right)_{n\ge k}$ is $q$-log-concave for any $d\ge 1$.
		\end{theorem}

    	\begin{proof}
		Recall that $H_q^{(d)}(k+m,k)$ equals
		\begin{equation*}
		\ps_{k+m}
		\left(s_{(d^k)}\right)
		=q^{\frac{dk(k-1)}{2}}
		\prod_{j=0}^{d-1}
		\frac{\begin{bmatrix}
			k+m+j\\ k
		\end{bmatrix}_q}{
		\begin{bmatrix}
			k+j\\ k
		\end{bmatrix}_q }=q^{\frac{dk(k-1)}2}
		\prod_{j=0}^{d-1}
		\frac{[k+m+j]_q![j]_q!}
		{[k+j]_q![m+j]_q!}.		    
		\end{equation*}
		By Lemma~\ref{lem:hook-content}, the hook-content formula gives
		\begin{equation*}
		\ps_{k+d}\left(s_{(m^d)}\right)
		=
		q^{m\binom{d}{2}}
		\prod_{i=1}^{d}\prod_{j=1}^{m}
		\frac{[k+d+j-i]_q}{[m+d-i-j+1]_q}.		
		\end{equation*}
		Reversing the index $i$ and rearranging the products, we obtain
		\begin{equation*}
		\ps_{k+d}\left(s_{(m^d)}\right)
		=
		q^{m\binom{d}{2}}
		\prod_{j=0}^{d-1}
		\frac{[k+m+j]_q![j]_q!}
		{[k+j]_q![m+j]_q!}.		
		\end{equation*}
		Therefore,
		\begin{equation*}
		H_q^{(d)}(k+m,k)
		=
		q^{\frac{dk(k-1)}2-m\binom{d}{2}}
		\ps_{k+d}
		\left(s_{(m^d)}\right).		    
		\end{equation*}
		It is obvious that
		\begin{equation*}
		H_q^{(d)}(k,k)^2=q^{dk(k-1)}\ge_q0.		
		\end{equation*}
		For $m\geq 1$, we have
		\begin{align*}
			& H_q^{(d)}(k+m,k)^2
			-
			H_q^{(d)}(k+m-1,k)H_q^{(d)}(k+m+1,k)
			\\
			={}&
			q^{dk(k-1)-d(d-1)m}
			\ps_{k+d}
			\left(
			s_{(m^d)}^2
			-
			s_{((m-1)^d)}s_{((m+1)^d)}
			\right).		    
		\end{align*}
	     Taking $a=m, b=d$ in~\eqref{kir}, $H_q^{(d)}(k+m,k)^2-
			H_q^{(d)}(k+m-1,k)H_q^{(d)}(k+m+1,k)$ equals
        \begin{equation*}
			q^{dk(k-1)-d(d-1)m}
			\ps_{k+d}
			\left(
			s_{(m^{d-1})}s_{(m^{d+1})}
			\right).            
        \end{equation*}
		
		By the same argument as in the proof of Lemma~\ref{lem:min-degree}, we can obtain 
		\begin{equation*}
		q^{-m\binom{d-1}{2}}
		\ps_{k+d}
		\left(s_{(m^{d-1})}\right)
		\geq_q0 
        \quad \textrm{and} \quad
        q^{-m\binom{d+1}{2}}
		\ps_{k+d}
		\left(s_{(m^{d+1})}\right)
		\geq_q0.		    
		\end{equation*}
		Consequently, $H_q^{(d)}(k+m,k)^2-H_q^{(d)}(k+m-1,k)H_q^{(d)}(k+m+1,k)$ equals
		\begin{equation*}
        q^{dk(k-1)+m}
		\left(q^{-m\binom{d-1}{2}}
		\ps_{k+d}
		\left(s_{(m^{d-1})}\right) \right) \left(q^{-m\binom{d+1}{2}}
		\ps_{k+d}
		\left(s_{(m^{d+1})}\right)\right).		    
		\end{equation*}
		Then for $d\ge 1$, we have
		\begin{equation*}
		H_q^{(d)}(k+m,k)^2
		-
		H_q^{(d)}(k+m-1,k)H_q^{(d)}(k+m+1,k)
		\geq_q0.		
		\end{equation*}
		Thus, for any fixed $k$, the sequence
		$\left(H_q^{(d)}(n,k)\right)_{n\geq k}$
		is $q$-log-concave.
	\end{proof}
		
	
		

\appendix
\renewcommand{\thetheorem}{\Alph{section}.\arabic{theorem}}
\section{Computational details for the proof of~\eqref{c1-c2}}
\begin{itemize}
    	\item[(iii)]  Evaluation of $A_2-D_2$.
		Recall that 
		$$
		A_2=\sum_{k=0}^{r}\sum_{0\le a\le b\le c\le2}
		\ps_{1}^{1}
		\left(
		s_{(3^a,2^{b-a},1^{c-b-1})}
		\right)
		\ps_{n-1}^{1}
		\!\left(
		s_{(3^{k-c},2^{c-b},1^{b-a})}
		\, s_{(3^{r-k})}
		\right).
		$$
		The term $\ps_1^1\left(s_{(3^a,2^{b-a},1^{c-b-1})}\right)=0$ unless $c-b-1\ge 0$.
		Hence, the double sum $A_2$ can be decomposed into four sums $A_2=A_{21}+A_{22}+A_{23}+A_{24}$,
		where 
		$$
		\begin{array}{c|c|c}
			(a,b,c)
			&
			(3^{k-c},2^{c-b},1^{b-a})
			&
			A_{2,i}\\ \hline
			(0,0,1)&(3^{k-1},2)&A_{21}=\sum_{k=0}^{r}
			\ps_{n-1}^1
			(s_{(3^{k-1},2)}\,
			s_{(3^{r-k})})\\
			(0,0,2)&(3^{k-2},2^2)&A_{22}=\sum_{k=0}^{r}
			\ps_{n-1}^1
			(s_{(3^{k-2},2^2)}\,
			s_{(3^{r-k})})\\
			(0,1,2)&(3^{k-2},2,1)&A_{23}=\sum_{k=0}^{r}
			\ps_{n-1}^1
			(s_{(3^{k-2},2,1)}\,
			s_{(3^{r-k})})\\
			(1,1,2)&(3^{k-2},2)&A_{24}=\sum_{k=0}^{r}
			\ps_{n-1}^1
			(s_{(3^{k-2},2)}\,
			s_{(3^{r-k})})
		\end{array}.
		$$
		Recall that 
		$$D_2=\sum_{k=0}^{r}
		\sum_{0\le a\le b\le c\le1}
		\ps_{1}^{1}\!\left(s_{(3^{a},2^{b-a},1^{c-b})}\right)
		\ps_{n-1}^{1} \Big(
		s_{(3^{k-c},2^{c-b},1^{b-a})}
		\,
		s_{(3^{r-k-1},2)}
		\Big)
		$$
		and $\ell(3^a,2^{b-a},1^{c-b})=c$.
		The term $\ps_1^1\left(s_{(3^a,2^{b-a},1^{c-b})}\right)=0$  unless $c\le 1$.
		Hence, the double sum $D_2$ can be decomposed into four sums $D_2=D_{21}+D_{22}+D_{23}+D_{24},$
		where 
		$$
		\begin{array}{c|c|c}
			(a,b,c)
			&
			(3^{k-c},2^{c-b},1^{b-a})
			&
			D_{2,i}\\ \hline
			(0,0,0)&(3^{k})&D_{21}=\sum_{k=0}^{r}
			\ps_{n-1}^1
			(s_{(3^{k})}\,
			s_{(3^{r-k-1},2)})\\
			(0,0,1)&(3^{k-1},2)&D_{22}=\sum_{k=0}^{r}
			\ps_{n-1}^1
			(s_{(3^{k-1},2)}\,
			s_{(3^{r-k-1},2)})\\
			(0,1,1)&(3^{k-1},1)&D_{23}=\sum_{k=0}^{r}
			\ps_{n-1}^1
			(s_{(3^{k-1},1)}\,
			s_{(3^{r-k-1},2)})\\
			(1,1,1)&(3^{k-1})&D_{24}=\sum_{k=0}^{r}
			\ps_{n-1}^1
			(s_{(3^{k-1})}\,
			s_{(3^{r-k-1},2)})
		\end{array}.
		$$
		A direct computation yields 
		\begin{equation*}
			A_{21}=\sum_{k=0}^{r}
			\ps_{n-1}^1
			\left(s_{(3^{k-1},2)}\,
			s_{(3^{r-k})}\right)
			=\sum_{k=0}^{r}
			\ps_{n-1}^1
			\left(s_{(3^{r-k-1},2)}\,
			s_{(3^{k})}\right)=D_{21},
		\end{equation*}
        and
		\begin{align*}
			D_{24}&=\sum_{k=0}^{r}
			\ps_{n-1}^1
			\Big(s_{(3^{k-1})}\,
			s_{(3^{r-k-1},2)}\Big)=\sum_{k=0}^{r}
			\ps_{n-1}^1
			\Big(s_{(3^{r-k-1})}\,
			s_{(3^{k-1},2)}\Big)\\
			&=\sum_{k=1}^{r-1}
			\ps_{n-1}^1
			\Big(s_{(3^{r-k-1})}\,
			s_{(3^{k-1},2)}\Big)=\sum_{k=0}^{r-2}
			\ps_{n-1}^1
			\Big(s_{(3^{r-k-2})}\,
			s_{(3^{k},2)}\Big)\\
			&=\sum_{k=2}^{r}
			\ps_{n-1}^1
			\Big(s_{(3^{r-k})}\,
			s_{(3^{k-2},2)}\Big)=\sum_{k=0}^{r}
			\ps_{n-1}^1
			\Big(s_{(3^{r-k})}\,
			s_{(3^{k-2},2)}\Big)=A_{24}.
		\end{align*}
		Hence, $A_2-D_2=A_{22}+A_{23}-D_{22}-D_{23}$ equals
		\begin{align*}
			&\sum_{k=0}^{r}\ps_{n-1}^1\Big(
			s_{(3^{k-2},2^2)}\, s_{(3^{r-k})}
			+s_{(3^{k-2},2,1)}\, s_{(3^{r-k})}\\
			&\qquad -s_{(3^{k-1},2)}\, s_{(3^{r-k-1},2)}
			-s_{(3^{k-1},1)}\, s_{(3^{r-k-1},2)}
			\Big).
		\end{align*}
		\item[(iv)]  Evaluation of $A_3+A_5-D_3.$
		Recall that 
		$$
		A_3=\sum_{k=0}^{r}\sum_{0\le a\le b\le c\le2}
		\ps_{1}^{1}
		\left(
		s_{(3^a,2^{b-a-1},1^{c-b+1})}
		\right)
		\ps_{n-1}^{1}
		\!\left(
		s_{(3^{k-c},2^{c-b},1^{b-a})}
		\, s_{(3^{r-k})}
		\right),
		$$
		and $\ell(3^a,2^{b-a-1},1^{c-b+1})=c$.
		The term $\ps_1^1\left(s_{(3^a,2^{b-a-1},1^{c-b+1})}\right)=0$ unless $c\le 1$ and $b-a-1\ge 0$,
		which implies that $a=0,b=c=1$. Thus, 
		$$A_3=
		\sum_{k=0}^{r}
		\ps_{n-1}^1
		\Big(
		s_{(3^{k-1},1)}\, s_{(3^{r-k})}\Big).
		$$
		Recall that 
		$$
		A_5=\sum_{k=0}^{r}\sum_{0\le a\le b\le c\le2}
		\ps_{1}^{1}
		\left(
		s_{(3^a,2^{b-a-1},1^{c-b})}
		\right)
		\ps_{n-1}^{1}
		\!\left(
		s_{(3^{k-c},2^{c-b},1^{b-a})}
		\, s_{(3^{r-k})}
		\right),
		$$
		and $\ell(3^a,2^{b-a-1},1^{c-b})=c-1$.
		The term $\ps_1^1\left(s_{(3^a,2^{b-a-1},1^{c-b})}\right)=0$ unless $b-a-1\ge 0$.
		Hence,  $A_5$ can be decomposed into four sums $A_5=A_{51}+A_{52}+A_{53}+A_{54}$,
		where 
		$$
		\begin{array}{c|c|c}
			(a,b,c)
			&
			(3^{k-c},2^{c-b},1^{b-a})
			&
			A_{5,i}\\ \hline
			(0,1,1)&(3^{k-1},1)&A_{51}=\sum_{k=0}^{r}
			\ps_{n-1}^1
			(s_{(3^{k-1},1)}\,
			s_{(3^{r-k})})\\
			(0,1,2)&(3^{k-2},2,1)&A_{52}=\sum_{k=0}^{r}
			\ps_{n-1}^1
			(s_{(3^{k-2},2,1)}\,
			s_{(3^{r-k})})\\
			(0,2,2)&(3^{k-2},1^2)&A_{53}=\sum_{k=0}^{r}
			\ps_{n-1}^1
			(s_{(3^{k-2},1^2)}\,
			s_{(3^{r-k})})\\
			(1,2,2)&(3^{k-2},1)&A_{54}=\sum_{k=0}^{r}
			\ps_{n-1}^1
			(s_{(3^{k-2},1)}\,
			s_{(3^{r-k})})
		\end{array}.
		$$
		Recall that 
		$$D_3=\sum_{k=0}^{r}
		\sum_{0\le a\le b\le c\le1}
		\ps_{1}^{1}\!\left(s_{(3^{a},2^{b-a},1^{c-b})}\right)
		\ps_{n-1}^{1} \Big(
		s_{(3^{k-c},2^{c-b},1^{b-a})}
		\,
		s_{(3^{r-k-1},1)}
		\Big),
		$$
		and $\ell(3^a,2^{b-a},1^{c-b})=c$.
		The term $\ps_1^1\left(s_{(3^a,2^{b-a},1^{c-b})}\right)=0$  unless $c\le 1$.
		Hence,  $D_3$ can be decomposed into four sums $D_3=D_{31}+D_{32}+D_{33}+D_{34},$
		where 
		$$
		\begin{array}{c|c|c}
			(a,b,c)
			&
			(3^{k-c},2^{c-b},1^{b-a})
			&
			D_{3,i}\\ \hline
			(0,0,0)&(3^{k})&D_{31}=\sum_{k=0}^{r}
			\ps_{n-1}^1
			(s_{(3^{k})}\,
			s_{(3^{r-k-1},1)})\\
			(0,0,1)&(3^{k-1},2)&D_{32}=\sum_{k=0}^{r}
			\ps_{n-1}^1
			(s_{(3^{k-1},2)}\,
			s_{(3^{r-k-1},1)})\\
			(0,1,1)&(3^{k-1},1)&D_{33}=\sum_{k=0}^{r}
			\ps_{n-1}^1
			(s_{(3^{k-1},1)}\,
			s_{(3^{r-k-1},1)})\\
			(1,1,1)&(3^{k-1})&D_{34}=\sum_{k=0}^{r}
			\ps_{n-1}^1
			(s_{(3^{k-1})}\,
			s_{(3^{r-k-1},1)})
		\end{array}.
		$$
		A direct computation yields 
		$$A_3=
		\sum_{k=0}^{r}
		\ps_{n-1}^1
		\Big(
		s_{(3^{k-1},1)}\, s_{(3^{r-k})}\Big)=
		\sum_{k=0}^{r}
		\ps_{n-1}^1
		\Big(
		s_{(3^{r-k-1},1)}\, s_{(3^{k})}\Big)=D_{31},
		$$
        and
		\begin{align*}
			A_{54}&=\sum_{k=0}^{r}
			\ps_{n-1}^1\Big(
			s_{(3^{k-2},1)}\,
			s_{(3^{r-k})}\Big)=\sum_{m=1}^{r+1}
			\ps_{n-1}^1\Big(
			s_{(3^{r-m-1},1)}\,
			s_{(3^{m-1})}\Big)\\
			&=\sum_{m=0}^{r}
			\ps_{n-1}^1\Big(
			s_{(3^{r-m-1},1)}\,
			s_{(3^{m-1})}\Big)=D_{34}.
		\end{align*}
		Hence, $A_3+A_5-D_3=A_{51}+A_{52}+A_{53}-D_{32}-D_{33}$ equals
		\begin{align*}
			&\sum_{k=0}^{r}\ps_{n-1}^1\Big(
			s_{(3^{k-1},1)}\, s_{(3^{r-k})}
			+s_{(3^{k-2},2,1)}\, s_{(3^{r-k})}
			+s_{(3^{k-2},1^2)}\, s_{(3^{r-k})}\\
			&-s_{(3^{k-1},2)}\, s_{(3^{r-k-1},1)}
			-s_{(3^{k-1},1)}\, s_{(3^{r-k-1},1)}
			\Big).
		\end{align*}
		\item[(v)] Evaluation of $A_4+A_6+A_7$.
		Recall that 
		$$
		A_4=\sum_{k=0}^{r}\sum_{0\le a\le b\le c\le2}
		\ps_{1}^{1}
		\left(
		s_{(3^{a-1},2^{b-a+1},1^{c-b})}
		\right)
		\ps_{n-1}^{1}
		\!\left(
		s_{(3^{k-c},2^{c-b},1^{b-a})}
		\, s_{(3^{r-k})}
		\right),
		$$
		and $\ell(3^{a-1},2^{b-a+1},1^{c-b})=c$.
		The term $\ps_1^1\left(s_{(3^{a-1},2^{b-a+1},1^{c-b})}\right)=0$ unless $c\le 1$ and $a-1\ge 0$,
		which implies that $a=b=c=1$. Thus, 
		$$
		A_4=
		\sum_{k=0}^{r}
		\ps_{n-1}^1
		\Big(
		s_{(3^{k-1})} \cdot s_{(3^{r-k})}\Big).
		$$
		Recall that 
		$$
		A_6=\sum_{k=0}^{r}\sum_{0\le a\le b\le c\le2}
		\ps_{1}^{1}
		\left(
		s_{(3^{a-1},2^{b-a+1},1^{c-b-1})}
		\right)
		\ps_{n-1}^{1}
		\!\left(
		s_{(3^{k-c},2^{c-b},1^{b-a})}
		\cdot s_{(3^{r-k})}
		\right),
		$$
		and $\ell(3^{a-1},2^{b-a+1},1^{c-b-1})=c-1$.
		The term $\ps_1^1\left(s_{(3^{a-1},2^{b-a+1},1^{c-b-1})}\right)=0$ unless  $a-1\ge 0$ and $c-b-1\ge 0$,
		which implies that $a=1,b=1,c=2$. Thus, 
		$$A_6=
		\sum_{k=0}^{r}
		\ps_{n-1}^1
		\Big(
		s_{(3^{k-2},2)} \, s_{(3^{r-k})}\Big).
		$$
		Recall that 
		$$
		A_7=\sum_{k=0}^{r}\sum_{0\le a\le b\le c\le2}
		\ps_{1}^{1}
		\left(
		s_{(3^{a-1},2^{b-a},1^{c-b+1})}
		\right)
		\ps_{n-1}^{1}
		\!\left(
		s_{(3^{k-c},2^{c-b},1^{b-a})}
		\, s_{(3^{r-k})}
		\right),
		$$
		and $\ell(3^{a-1},2^{b-a},1^{c-b+1})=c$.
		The term $\ps_1^1\left(s_{(3^{a-1},2^{b-a},1^{c-b+1})}\right)=0$ unless $c\le 1$ and $a-1\ge 0$,
		which implies that $a=b=c=1$. Thus, 
		$$
		A_7=
		\sum_{k=0}^{r}
		\ps_{n-1}^1
		\Big(
		s_{(3^{k-1})}\,s_{(3^{r-k})}\Big).
		$$
		Hence,  
		$$
		A_4+A_6+A_7
		=
		\sum_{k=0}^{r}
		\ps_{n-1}^1
		\Big(
		2s_{(3^{k-1})}\, s_{(3^{r-k})}
		+s_{(3^{k-2},2)} \, s_{(3^{r-k})}
		\Big).
		$$
\end{itemize}

\section{Auxiliary symmetric-function identities}
\begin{lemma}\label{proof_j2}
We have
\begin{align}
e_2 \Ya-e_1 \Yb&=\Zc-\Yc \label{eq-proof-j2},\\
e_2 \Yb-e_1\Yc&=\Ze+\frac{\Za}{t}-\Zf.\label{eq-proof-j21}
\end{align}
\end{lemma}
\begin{proof}
		By vertical Pieri's rule (Theorem~\ref{h-e}), we have
		$$
		e_2s_{(k,k)}
		=s_{(k+1,k+1)}
		+s_{(k+1,k,1)}
		+s_{(k,k,1,1)},
		$$
		$$
		e_1s_{(k+1,k)}
		=s_{(k+2,k)}
		+s_{(k+1,k+1)}
		+s_{(k+1,k,1)}.
		$$
		Subtracting these two identities, multiplying by $t^k$ and summing over
        $k\ge 0$ yields
		$$
		e_2\Ya-e_1\Yb=\Zc-\Yc.
		$$
		Similarly,
		$$
		e_2s_{(k+1,k)}
		=s_{(k+2,k+1)}
		+s_{(k+2,k,1)}
		+s_{(k+1,k+1,1)}
		+s_{(k+1,k,1,1)},
		$$
		$$
		e_1s_{(k+2,k)}
		=s_{(k+3,k)}
		+s_{(k+2,k+1)}
		+s_{(k+2,k,1)}.
		$$
		Therefore,
		$$
		e_2s_{(k+1,k)}-e_1s_{(k+2,k)}
		=
		s_{(k+1,k+1,1)}
		+s_{(k+1,k,1,1)}
		-s_{(k+3,k)}.
		$$
		Multiplying by $t^k$ and summing over $k\geq0$, and using 
		${\Za}/{t}
		=
		\sum_{k\geq0}s_{(k+1,k+1,1)}\,t^k,
		$
		we have
		$$
		e_2\Yb-e_1\Yc=\frac{\Za}{t}+\Ze-\Zf.
		$$
This completes the proof.
\end{proof}

		We present two further identities that will be used in the proof of~\eqref{eq:base-pf3}.
		\begin{lemma}\label{lem-eq19}
		We have
			\begin{equation}\label{eq:J2-c}
			\Zf=e_3\Ya-e_2\Yb+e_1\Yc-S_{(k,k,1,1,1)}.
			\end{equation}
            \begin{align}
			\Zg&=\frac{\Ya-1}{t^2}+\frac{h_2\Ya-e_1\Yb}{t}+e_2\Yc\label{eq:J2-d}\\
			&\qquad -(e_1h_2-h_3)\cdot\Yb+(h_2^2-e_1h_3)\cdot\Ya. \nonumber
			\end{align}
		\end{lemma}
		\begin{proof}
			By vertical Pieri's rule (Theorem~\ref{h-e}),
			$$
			e_3s_{(k,k)}
			=s_{(k+1,k+1,1)}
			+s_{(k+1,k,1,1)}
			+s_{(k,k,1,1,1)},\quad 
			e_1s_{(k+2,k)}
			=s_{(k+3,k)}
			+s_{(k+2,k+1)}
			+s_{(k+2,k,1)},
			$$
			$$
			e_2s_{(k+1,k)}
			=s_{(k+2,k+1)}
			+s_{(k+2,k,1)}
			+s_{(k+1,k+1,1)}
			+s_{(k+1,k,1,1)}.
			$$
			Thus,
			$$
			e_3s_{(k,k)}-e_2s_{(k+1,k)}+e_1s_{(k+2,k)}
			=s_{(k+3,k)}+s_{(k,k,1,1,1)}
			$$
		    and equation~\eqref{eq:J2-c} follows.

			To prove~\eqref{eq:J2-d}, let $R$ denote the right-hand side of~\eqref{eq:J2-d}.
			For $k\geq 0$,
			\begin{align*}
				[t^k]R
				={}&s_{(k+2,k+2)}
				+h_2s_{(k+1,k+1)}
				-e_1s_{(k+2,k+1)}
				+e_2s_{(k+2,k)}\\
				&-(e_1h_2-h_3)s_{(k+1,k)}
				+(h_2^2-e_1h_3)s_{(k,k)}.
			\end{align*}
			By Pieri's rule,
			$$
			h_2s_{(k+1,k+1)}-e_1s_{(k+2,k+1)}
			=s_{(k+1,k+1,2)}-s_{(k+2,k+2)},
			$$
			and
			$$
			e_2s_{(k+2,k)}
			=s_{(k+3,k+1)}
			+s_{(k+3,k,1)}
			+s_{(k+2,k+1,1)}
			+s_{(k+2,k,1,1)}.
			$$
			Moreover, since
			$$
			e_1h_2-h_3=s_{(2,1)},
			\quad
			h_2^2-e_1h_3=s_{(2,2)},
			$$
			repeated applications of Pieri's rule give
			\begin{align*}
				(e_1h_2-h_3)s_{(k+1,k)}
				={}&s_{(k+3,k+1)}
				+s_{(k+3,k,1)}
				+s_{(k+2,k+2)}
				+2s_{(k+2,k+1,1)}
				\\
				&+s_{(k+2,k,2)}
				+s_{(k+2,k,1,1)}
				+s_{(k+1,k+1,2)}
				+s_{(k+1,k+1,1,1)}
				+s_{(k+1,k,2,1)},
			\end{align*}
			whereas
			\begin{align*}
				(h_2^2-e_1h_3)s_{(k,k)}
				={}&s_{(k+2,k+2)}
				+s_{(k+2,k+1,1)}
				+s_{(k+2,k,2)}
				\\
				&+s_{(k+1,k+1,1,1)}
				+s_{(k+1,k,2,1)}
				+s_{(k,k,2,2)}.
			\end{align*}
			Substitution and cancellation yield
			$
			[t^k]R=s_{(k,k,2,2)}.
			$
		This completes the proof.
		\end{proof}
		
	

\begin{proof}[Proof of~\eqref{eq:base-pf3}]
Recall that
		\begin{align*}
        &M_{12}=\frac{\Ya-1-e_2t-(e_2^2-e_1e_3)t^2}{t^3},\\
			&M_{13}=\frac{\Yb-e_1-(e_1e_2-e_3)t}{t^2},\quad
			M_{14}=\frac{\Yc-h_2}{t},\quad
			M_{15}=\Zf,\\
			&M_{23}=\frac{\Ya-1-e_2t}{t^2},\quad
			M_{24}=\frac{\Yb-e_1}{t},\quad
			M_{25}=\Yc,\\
			&M_{34}=\frac{\Ya-1}{t},\quad
			M_{35}=\Yb,\quad
			M_{45}=\Ya.
		\end{align*}
It remains to prove
\begin{align*}
&h_2\Pf(M_{\{2,3,4,5\}, \{2,3,4,5\}})-e_1\Pf(M_{\{1,3,4,5\},\{1,3,4,5\}})+\Pf(M_{\{1,2,4,5\},\{1,2,4,5\}})\\
={}&\frac{\Ya\Zg-\Zc^2+S_{(k,k,1)}S_{(k,k,1,1,1)}}{t}.
\end{align*}
Denote $L=h_2\Pf(M_{\{2,3,4,5\}, \{2,3,4,5\}})-e_1\Pf(M_{\{1,3,4,5\},\{1,3,4,5\}})+\Pf(M_{\{1,2,4,5\},\{1,2,4,5\}}).$
We have
		\begin{align*}
			&\Pf\left(
			M_{\{2,3,4,5\},\{2,3,4,5\}}
			\right)
			=M_{23}M_{45}-M_{24}M_{35}+M_{25}M_{34}\\
			={}&\frac{\Ya-1-e_2t}{t^2}\Ya
			-\frac{\Yb-e_1}{t}\Yb
			+\Yc\frac{\Ya-1}{t}\\
			={}&\frac{\Ya^2-\Ya}{t^2}+
			\frac{-e_2\Ya
				-\Yb^2+e_1\Yb+\Ya\Yc- \Yc}{t},
		\end{align*}
		\begin{align*}
			&\Pf\left(
			M_{\{1,3,4,5\},\{1,3,4,5\}}
			\right)
			=M_{13}M_{45}-M_{14}M_{35}+M_{15}M_{34}\\
			={}&\frac{\Yb-e_1-(e_1e_2-e_3)t}{t^2}\Ya
			-\frac{\Yc-h_2}{t}\Yb
			+\Zf\frac{\Ya-1}{t}\\
			={}&\frac{\Ya\Yb-e_1\Ya}{t^2}\\
			&+\frac{-(e_1e_2-e_3)\Ya-\Yb\Yc+h_2\Yb+\Zf\Ya-\Zf}{t},
		\end{align*}
		and
		\begin{align*}
			&\Pf\left(
			M_{\{1,2,4,5\},\{1,2,4,5\}}
			\right)
			=M_{12}M_{45}-M_{14}M_{25}+M_{15}M_{24}\\
			={}&\frac{\Ya-1-e_2t-(e_2^2-e_1e_3)t^2}{t^3}\Ya
			-\frac{\Yc-h_2}{t}\Yc
			+\Zf\frac{\Yb-e_1}{t}\\
			={}&\frac{\Ya^2-\Ya}{t^3}+\frac{-e_2\Ya}{t^2}\\
			&+\frac{-(e_2^2-e_1e_3)\Ya-\Yc^2+h_2\Yc+\Yb\Zf-e_1\Zf}{t}.
		\end{align*}
Collecting terms gives
		\begin{align*}
        tL&=t\left(h_2\Pf(M_{\{2,3,4,5\}, \{2,3,4,5\}})-
		e_1\Pf(M_{\{1,3,4,5\},\{1,3,4,5\}})+\Pf(M_{\{1,2,4,5\},\{1,2,4,5\}})\right)\\
			&=\frac{\Ya^2-\Ya}{t^2}
			+\frac{h_2\Ya^2-e_1\Ya\Yb
				+\Ya(e_1^2-e_2-h_2)}{t}\\
			&\quad+e_2\Ya(e_1^2-e_2-h_2)+h_2\Ya\Yc\\
			&\quad+e_1\Yb\Yc-h_2\Yb^2
			-\Yc^2+\left(\Yb-e_1\Ya\right)\Zf.
		\end{align*}
		Since
		$
		h_2=e_1^2-e_2,
		$
		\begin{align*}
		tL=&\frac{\Ya^2-\Ya}{t^2}
			+\frac{h_2\Ya^2-e_1\Ya\Yb}{t}+h_2\Ya\Yc\\
			&+e_1\Yb\Yc-h_2\Yb^2
			-\Yc^2+\left(\Yb-e_1\Ya\right)\Zf.
		\end{align*}
		By~\eqref{eq:J2-d} in Lemma~\ref{lem-eq19}, $\Zg$ equals
        $$
        \frac{\Ya-1}{t^2}+\frac{h_2\Ya-e_1\Yb}{t}+e_2\Yc
	 -(e_1h_2-h_3)\cdot\Yb+(h_2^2-e_1h_3)\cdot\Ya.
        $$
		It follows that $tL-\Ya\Zg$ equals
		\begin{align*}
			&-(h_2^2-e_1h_3)\Ya^2-h_2\Yb^2-\Yc^2
			+(e_1h_2-h_3)\Ya\Yb\\
            &\quad +(h_2-e_2)\Ya\Yc
	    +e_1\Yb\Yc+(\Yb-e_1\Ya)\Zf.
		\end{align*}
		Furthermore, equation~\eqref{eq-proof-j2} shows that $e_2 \Ya-e_1 \Yb=\Zc-\Yc$. Hence,
		\begin{align*}
			\Zc^2
			&=e_2^2\Ya^2+e_1^2\Yb^2+\Yc^2\\
			&\quad -2e_1e_2\Ya\Yb
			+2e_2\Ya\Yc-2e_1\Yb\Yc.
		\end{align*}
Thus, $tL-\Ya\Zg+\Zc^2$ equals
		\begin{align*}
			&\bigl(-h_2^2+e_1h_3+e_2^2\bigr)\Ya^2
	  +(-h_2+e_1^2)\Yb^2+\bigl(e_1h_2-h_3-2e_1e_2\bigr)\Ya\Yb\\
			&\quad+(h_2-e_2+2e_2)\Ya\Yc
+(e_1-2e_1)\Yb\Yc+(\Yb-e_1\Ya)\Zf.
		\end{align*}
		Using the identities $h_2=e_1^2-e_2,$ $e_1h_2-h_3=e_1e_2-e_3,$ and $h_2^2-e_1h_3=e_2^2-e_1e_3,$  we have
		\begin{align*}
			&tL-\Ya\Zg+\Zc^2\\
			={}&e_1e_3\Ya^2+e_2\Yb^2
			-(e_1e_2+e_3)\Ya\Yb\\
			&\quad +e_1^2\Ya\Yc-e_1\Yb\Yc+\left(\Yb-e_1\Ya\right)\cdot \Zf\\
			=&\left(\Yb-e_1\Ya\right)\cdot
			\left(-e_3\Ya+e_2\Yb-e_1\Yc\right)+\left(\Yb-e_1\Ya\right)\cdot \Zf\\
			=&\left(\Yb-e_1\Ya\right)\cdot
			\left(\Zf-e_3\Ya+e_2\Yb-e_1\Yc\right).
		\end{align*}
		Hence, we obtain 
		\begin{align*}
			&tL
			-\Ya\Zg+\Zc^2+\left(S_{(k+1,k)}-e_1S_{(k,k)}\right)\cdot S_{(k,k,1,1,1)}\\
			=&\left(\Yb-e_1\Ya\right) \cdot
			\left(
			\Zf-e_3\Ya+e_2\Yb-e_1\Yc+ S_{(k,k,1,1,1)}
			\right).
		\end{align*}
		By~\eqref{eq:J2-c}, the right-hand side vanishes.
        Since $e_1S_{(k,k)}=S_{(k+1,k)}+S_{(k,k,1)}$,
        we derive that
		\begin{align*}
			tL
			-\Ya\Zg+\Zc^2- S_{(k,k,1)}S_{(k,k,1,1,1)}=0.
		\end{align*}
This completes the proof.
\end{proof}

\section*{Declaration of AI usage}
Some algebraic simplifications in Appendix~B were carried out
with assistance of ChatGPT 5.6 Pro. The proof strategy and all
mathematical arguments are due to the authors. The authors
independently verified all AI-assisted computations and take full
responsibility for the content of the manuscript.
\section*{Acknowledgements}
This work was supported in part by the National Natural Science Foundation of China (No. 12201100).




\begin{thebibliography}{99}	
\addcontentsline{toc}{section}{References}

	\bibitem{Bax64}
	G. Baxter,
	\textit{On fixed points of the composite of commuting functions},
	Proc. Amer. Math. Soc. \textbf{15} (1964), 851--855.
	
	\bibitem{Bha99}
	G. Bhatnagar,
	\textit{A short proof of an identity of Sylvester},
	Int. J. Math. Math. Sci. \textbf{22} (1999), 431--435.
	
	\bibitem{Bra04}
	P. Br{\"a}nd{\'e}n,
	\textit{$q$-Narayana numbers and the flag $h$-vector of
		$J(\mathbf{2}\times\mathbf{n})$},
	Discrete Math. \textbf{281} (2004), 67--81.
	
	\bibitem{Bra15}
	P. Br{\"a}nd{\'e}n,
	\textit{Unimodality, log-concavity, real-rootedness and beyond},
	in \textit{Handbook of Enumerative Combinatorics},
	Discrete Math. Appl. (Boca Raton),
	CRC Press, Boca Raton, FL, 2015, pp.~437--483.
	
	\bibitem{BD02}
	D. Bump and P. Diaconis,
	\textit{Toeplitz minors},
	J. Combin. Theory Ser. A \textbf{97} (2002), 252--271.
	
	
	\bibitem{CGH+24}
	L. Chen, A. Gibney, L. Heller, E. Kalashnikov, H. Larson and W. Xu,
	\textit{On an equivalence of divisors on $\overline{M}_{0,n}$
		from Gromov--Witten theory and conformal blocks},
	Transform. Groups \textbf{29} (2024), 561--590.
	
	\bibitem{CWY10}
	W.Y.C. Chen, L.X.W. Wang and A.L.B. Yang,
	\textit{Schur positivity and the $q$-log-convexity of the Narayana
		polynomials},
	J. Algebraic Combin. \textbf{32} (2010), 303--338.
	
	\bibitem{CWY11}
	W.Y.C. Chen, L.X.W. Wang and A.L.B. Yang,
	\textit{Recurrence relations for strongly $q$-log-convex polynomials},
	Canad. Math. Bull. \textbf{54} (2011), 217--229.
	
	\bibitem{CGHK78}
	F.R.K. Chung, R.L. Graham, V.E. Hoggatt Jr. and M. Kleiman,
	\textit{The number of Baxter permutations},
	J. Combin. Theory Ser. A \textbf{24} (1978), 382--394.
	
	\bibitem{Cig21}
	J. Cigler,
	\textit{Pascal triangle, Hoggatt matrices, and analogous constructions},
	arXiv:2103.01652, 2021.
	
	\bibitem{Dil12}
	K. Dilks,
	\textit{Involutions on Baxter objects},
	Discrete Math. Theor. Comput. Sci. Proc. AR (2012), 721--734.
	
	\bibitem{Dil15}
	K. Dilks,
	\textit{Involutions on Baxter Objects, and $q$-Gamma Nonnegativity},
	Ph.D. thesis, University of Minnesota, 2015.
	
	\bibitem{DV08}
	T. Do\v{s}li\'c and D. Veljan,
	\textit{Logarithmic behavior of some combinatorial sequences},
	Discrete Math. \textbf{308} (2008), 2182--2212.
	
	\bibitem{FZZ26}
	H. Fang, C.X.T. Zhang and J.J.Y. Zhao,
	\textit{Analytic properties arising from the Baxter numbers},
	J. Math. Anal. Appl. \textbf{561} (2026), Paper No.~130615.
	
	
	\bibitem{FK01}
	M. Fulmek and M. Kleber,
	\textit{Bijective proofs for Schur function identities which imply
		Dodgson's condensation formula and Pl\"ucker relations},
	Electron. J. Combin. \textbf{8} (2001), Research Paper 16.
	
	\bibitem{FH85}
	J. F{\"u}rlinger and J. Hofbauer,
	\textit{$q$-Catalan numbers},
	J. Combin. Theory Ser. A \textbf{40} (1985), 248--264.
	
	
	\bibitem{HIMN17}
	T. Hudson, T. Ikeda, T. Matsumura and H. Naruse,
	\textit{Degeneracy loci classes in $K$-theory---determinantal and
		Pfaffian formula},
	Adv. Math. \textbf{320} (2017), 115--156.
	
	\bibitem{HKKO25}
	J. Huh, J.S. Kim, C. Krattenthaler and S. Okada,
	\textit{Bounded Littlewood identities for cylindric Schur functions},
	Trans. Amer. Math. Soc. \textbf{378} (2025), 6765--6829.
	
	
	\bibitem{IN13}
	T. Ikeda and H. Naruse,
	\textit{$K$-theoretic analogues of factorial Schur $P$- and
		$Q$-functions},
	Adv. Math. \textbf{243} (2013), 22--66.
	
	\bibitem{IW95}
	M. Ishikawa and M. Wakayama,
	\textit{Minor summation formula of Pfaffians},
	Linear Multilinear Algebra \textbf{39} (1995), 285--305.
	
	\bibitem{IW96}
	M. Ishikawa, S. Okada and M. Wakayama,
	\textit{Applications of minor-summation formula I: Littlewood's formulas},
	J. Algebra \textbf{183} (1996), 193--216.
	
	\bibitem{IW99}
	M. Ishikawa and M. Wakayama,
	\textit{Applications of minor-summation formula II: Pfaffians and
		Schur polynomials},
	J. Combin. Theory Ser. A \textbf{88} (1999), 136--157.
	
	\bibitem{IW00}
	M. Ishikawa and M. Wakayama,
	\textit{Minor summation formulas of Pfaffians: Survey and a new identity},
	Adv. Stud. Pure Math. \textbf{28} (2000), 133--142.
	
	\bibitem{IW06}
	M. Ishikawa and M. Wakayama,
	\textit{Applications of minor summation formula III: Pl\"ucker
		relations, lattice paths and Pfaffian identities},
	J. Combin. Theory Ser. A \textbf{113} (2006), 113--155.
	
	\bibitem{JL15}
	N. Jing and Y. Li,
	\textit{A lift of Schur's $Q$-functions to the peak algebra},
	J. Combin. Theory Ser. A \textbf{135} (2015), 268--290.
	
	
	\bibitem{Kir84}
	A.N. Kirillov,
	\textit{Completeness of states of the generalized Heisenberg magnet},
	Zap. Nauchn. Sem. Leningrad. Otdel. Mat. Inst. Steklov. (LOMI)
	\textbf{134} (1984), 169--189.
	
	
	\bibitem{Knu96}
	D.E. Knuth,
	\textit{Overlapping Pfaffians},
	Electron. J. Combin. \textbf{3} (1996), Research Paper 5.
	
	
	\bibitem{LR12}
	S. Law and N. Reading,
	\textit{The Hopf algebra of diagonal rectangulations},
	J. Combin. Theory Ser. A \textbf{119} (2012), 788--824.
	
	\bibitem{LL23}
	Z. Lin and J. Liu,
	\textit{Proof of Dilks' bijectivity conjecture on Baxter permutations},
	J. Combin. Theory Ser. A \textbf{200} (2023), Paper No.~105796.
	
	\bibitem{LW07}
	L.L. Liu and Y. Wang,
	\textit{On the log-convexity of combinatorial sequences},
	Adv. in Appl. Math. \textbf{39} (2007), 453--476.
	
	\bibitem{Mac95}
	I.G. Macdonald,
	\textit{Symmetric Functions and Hall Polynomials},
	2nd ed., Oxford University Press, 1995.
	
	\bibitem{Mal79}
	C.L. Mallows,
	\textit{Baxter permutations rise again},
	J. Combin. Theory Ser. A \textbf{27} (1979), 394--396.
	
	\bibitem{MS26}
	J. Mao and W. Shi,
	\textit{The analytic properties of Hoggatt triangles},
	arXiv:2607.01582, 2026.
	
	
	
	\bibitem{Oka19}
	S. Okada,
	\textit{Pfaffian formulas and Schur $Q$-function identities},
	Adv. Math. \textbf{353} (2019), 446--470.
	
	\bibitem{Sag92_1}
	B.E. Sagan,
	\textit{Inductive proofs of $q$-log-concavity},
	Discrete Math. \textbf{99} (1992), 289--306.
	
	\bibitem{Sag92_2}
	B.E. Sagan,
	\textit{Log-concave sequences of symmetric functions and analogs of
		the Jacobi--Trudi determinants},
	Trans. Amer. Math. Soc. \textbf{329} (1992), 795--811.
	
	\bibitem{Sta71}
	R.P. Stanley,
	\textit{Theory and application of plane partitions. Part 2},
	Stud. Appl. Math. \textbf{50} (1971), 259--279.
	
	\bibitem{Sta89}
	R.P. Stanley,
	\textit{Log-concave and unimodal sequences in algebra, combinatorics,
		and geometry},
	Ann. New York Acad. Sci. \textbf{576} (1989), 500--535.
	
	\bibitem{Sta24}
	R.P. Stanley,
	\textit{Enumerative Combinatorics},
	Vol.~2, Cambridge University Press, Cambridge, 2024.
	
	\bibitem{Ste90}
	J.R. Stembridge,
	\textit{Nonintersecting paths, Pfaffians, and plane partitions},
	Adv. Math. \textbf{83} (1990), 96--131.
	
	\bibitem{Sun96}
	T.S. Sundquist,
	\textit{Two-variable Pfaffian identities and symmetric functions},
	J. Algebraic Combin. \textbf{5} (1996), 135--148.
	
	\bibitem{WZ16}
	Y. Wang and B.-X. Zhu,
	\textit{Log-convex and Stieltjes moment sequences},
	Adv. in Appl. Math. \textbf{81} (2016), 115--127.
	
	
	\bibitem{Zhu14}
	B.-X. Zhu,
	\textit{Some positivities in certain triangular arrays},
	Proc. Amer. Math. Soc. \textbf{142} (2014), 2943--2952.
	
	\bibitem{Zhu21}
	B.-X. Zhu,
	\textit{On a Stirling--Whitney--Riordan triangle},
	J. Algebraic Combin. \textbf{54} (2021), 999--1019.
	
	\bibitem{ZS15}
	B.-X. Zhu and H. Sun,
	\textit{Linear transformations preserving the strong
		$q$-log-convexity of polynomials},
	Electron. J. Combin. \textbf{22} (2015), Paper No.~3.26.
	
\end{thebibliography}
\end{document}